\documentclass[11pt,a4paper]{article}
\usepackage[utf8]{inputenc}
\usepackage[english]{babel}
\usepackage{amsmath}
\usepackage{amsfonts}
\usepackage{amssymb}
\usepackage{graphicx}
\usepackage{a4wide}
\usepackage{verbatim}
\RequirePackage{ifthen}
\RequirePackage[T1]{fontenc}
\RequirePackage{amsthm}
\RequirePackage{color}
\RequirePackage{listings} 
\RequirePackage{epstopdf}
\RequirePackage{emptypage}
\RequirePackage[nolist]{acronym}
\RequirePackage{multirow}
\RequirePackage{enumerate,enumitem}
\RequirePackage{caption}
\RequirePackage{tikz}
\usetikzlibrary{arrows.meta}
\RequirePackage{latexsym,amscd,makeidx}
\RequirePackage{exscale}
\RequirePackage{mathrsfs}
\RequirePackage{hyperref}
\RequirePackage{url}
\RequirePackage{booktabs}
\RequirePackage{esdiff}
\RequirePackage{listings}
\RequirePackage{tabularx}
\RequirePackage[numbers]{natbib}
\RequirePackage{verbatim}
\RequirePackage{comment}
\RequirePackage{cleveref}
\crefname{equation}{}{}
\Crefname{Equation}{}{}
\usepackage{pgfplots} 
\usepackage{pgfgantt}
\usepackage{pdflscape}
\pgfplotsset{compat=newest} 
\pgfplotsset{plot coordinates/math parser=false}
\numberwithin{equation}{section}
	\newif\ifshowcomments								
	\showcommentsfalse				

	\definecolor{mygrey}{gray}{0.65}
	
	\ifshowcomments
		
	\else
	 	\excludecomment{bcomment}
	\fi
	
	\ifshowcomments
		\newcommand{\scomment}[1]{{\color{mygrey} [#1]}}  
	\else
		\newcommand{\scomment}[1]{\iffalse #1 \fi}
	\fi

\newlength\fwidth

\DeclareMathOperator{\sgn}{sgn}
\DeclareMathOperator{\BV}{BV}

\theoremstyle{plain}						
\newtheorem{theorem}{Theorem}[section]
\newtheorem{lemma}[theorem]{Lemma}
\newtheorem{proposition}[theorem]{Proposition}

\theoremstyle{definition}
\newtheorem{definition}[theorem]{Definition}
\newtheorem{example}[theorem]{Example}
\newtheorem{assumption}[theorem]{Assumption}
\theoremstyle{remark}
\newtheorem{remark}[theorem]{Remark}

\newcommand{\Z}{\mathbb{Z}}
\newcommand{\N}{\mathbb{N}}

\newcommand\dy{\mathit{dy}}

\newcommand{\Ro}{\rho^{\max}_1}
\newcommand{\Rt}{\rho^{\max}_2}
\newcommand{\Rtr}{\rho^{\max}_3}
\newcommand{\Rr}{\rho^{\max}_e}
\newcommand{\po}{\rho_1}
\newcommand{\pt}{\rho_2}
\newcommand{\pe}{\rho_e}

\newcommand{\Dx}{\Delta x}
\newcommand{\Dt}{\Delta t}

\newcommand{\norm}[1]{{\left\|#1\right\|}}
\newcommand{\abs}[1]{{\left|#1\right|}}

\newcommand{\ndt}{{\eta}}
\newcommand{\wt}{\omega_{\ndt}}
\newcommand{\Ne}{N_\ndt}

\newcommand{\coupling}{g}

\newcommand{\testf}{\phi}

\newcommand{\dr}{\rho^{\Dx}}

\begin{document}

\title{Network models for nonlocal traffic flow}
\author{Jan Friedrich\footnotemark[1], \; Simone Göttlich\footnotemark[1], \; Maximilian Osztfalk\footnotemark[1]}

\date{\today}

\maketitle
\footnotetext[1]{University of Mannheim, Department of Mathematics, 68131 Mannheim, Germany (\{jan.friedrich, goettlich\}@uni-mannheim.de, max.osztfalk@gmail.com ).}

\begin{abstract}
We present a network formulation for a traffic flow model with nonlocal velocity in the flux function.
The modeling framework includes suitable coupling conditions 
at intersections to either ensure maximum flux or distribution parameters. Based on an upwind type numerical scheme, we prove the maximum principle and the existence of weak solutions on networks.
We also investigate the limiting behavior of the proposed models when the nonlocal influence tends to infinity.
Numerical examples show the difference between the proposed coupling conditions and a comparison to the Lighthill-Whitham-Richards network model. 

  \medskip


  \medskip

\end{abstract}

\begin{acronym}
\acro{BV}{bounded variation}
\acro{CFL}{Courant-Friedrichs-Levy}
\acro{ENO}{essential non oscillatory}
\acro{LWR}{Lighthill-Whitham-Richards}
\acro{LxF}{Lax-Friedrichs}
\acro{ODE}{ordinary differential equation}
\acro{PDE}{partial differential equation}
\acro{TVD}{total variation diminishing}
\acro{WENO}{weighted essential non oscillatory}
\acro{CWENO}{compact weighted essential non oscillatory}
\end{acronym}


\section{Introduction}
Macroscopic traffic flow models have been studied by researchers for several decades.
All started with the famous Lighthill-Whitham-Richards (LWR) \cite{LighthillWhitham, Richards} model, which has been introduced in the 1950's.
Since then the approach of modeling traffic flow by  conservation law has been extended in many directions.
A second equation has been introduced to describe the evolution of the speed and include acceleration \cite{aw2000resurrection,ZHANG2002traffic}. The LWR model has been further adapted to multilane traffic flow \cite{colombo2006multilane, greenberg2003congestion, helbing1997modeling, herty2003modeling, HoldenRisebro2019dense, moridpour2010surveymultilane} and complex road networks \cite{coclite2005network,garavellohanpiccoli2016book,GaravelloPiccoliBook,holden1995network, herty2003modeling}.
From an application point of view, macroscopic models can be used to investigate the so-called capacity drop effect and to optimize traffic flow \cite{colombo2011modelling,santo2019capacity, goatin2016speed,HautBastinChitour,herty2003modeling,kolb2017capacity}.

Nowadays, with the progress in autonomous driving, new challenges in road traffic arise. 
In order to manage these challenges mathematically, nonlocal traffic flow models have been introduced \cite{BlandinGoatin2016,ChiarelloGoatin,forcadel2021nonlocal, friedrich2018godunov, GoatinScialanga}.
They include more information in a certain nonlocal range about the traffic of the road.
This nonlocal range can stand for the connection radius of autonomous cars or for the sight of a driver.
Nonlocal models for traffic flow are widely studied in current research concerning existence of solutions \cite{BlandinGoatin2016,ChiarelloGoatin,friedrich2018godunov, GoatinScialanga,KeimerPflug2017}, numerical schemes \cite{BlandinGoatin2016,  chalons2018high, friedrich2019maximum, friedrich2018godunov, GoatinScialanga} or convergence to local conservation laws  \cite{bressan2020entropy,bressan2020traffic,coclite2020general,colombo2020local,keimer2019local} - even, in general, this question is still an open problem. 
Modeling approaches include microscopic models \cite{friedrich2020micromacro, shen2019stationary,GoatinRossi2017, RidderShen}, second order models \cite{friedrich2020micromacro}, multiclass models \cite{ChiarelloGoatinmulticlass}, multilane models \cite{BayenKeimer-preprint,friedrich2020nonlocal} and also time delay models \cite{keimer2019nonlocal}.
But to the best of our knowledge only a few works deal with network models.

Some first attempts can be found in \cite{camilli2018measure,chiarelloFriedrichGoatinGK, shen2019stationary, keimer2018nonlocal}.
The work \cite{camilli2018measure} considers measure valued solutions for nonlocal transport equations and \cite{keimer2018nonlocal} deals with nonlocal conservation laws on bounded domains while \cite{chiarelloFriedrichGoatinGK, shen2019stationary} includes 1-to-1 junctions.
In \cite{chiarelloFriedrichGoatinGK}, the existence and well-posedness of solutions at a 1-to-1 junction is shown, where the roads are allowed to differ in the velocity and maximum road capacities.
To the best of our knowledge, network models for other types of junctions have not been studied in literature and a general framework to deal with nonlocal traffic flow models on a network, similar to the LWR model, is missing in current research.

Hence, in this work we propose a network formulation for a class of nonlocal traffic flow models.
The work is thereby structured as following: in the next section we introduce the network setting. 
We concentrate on single junctions and define a weak solution. 
Additionally, we introduce rather general assumptions on the coupling conditions necessary to obtain a well-posed model.
It turns out that the 1-to-1 junction model from \cite{chiarelloFriedrichGoatinGK} fulfills these assumptions.
Furthermore, we also present two different explicit choices of the coupling conditions for 1-to-2 and 2-to-1 junctions which are inspired by well-studied couplings of the LWR model on a network.
In Section 3, we present a numerical scheme to solve the nonlocal equations on a network and argue that it converges to a weak solution.
In particular, we also prove that the numerical scheme and hence the weak solution inherit a suitable maximum principle.
Furthermore, we investigate the limit behavior of the junction models for a nonlocal range tending to infinity in Section 4. This case is of special interest for nonlocal traffic flow models since they are motivated by autonomous cars and an infinitely large range can be interpreted as prefect information for each driver.  
We close this work in Section 5 by numerically comparing the proposed coupling conditions for the nonlocal model.
On the one hand, we demonstrate differences between the nonlocal approaches and on the other hand, we give a comparison to the LWR network model.

\section{Network modeling}
Following the ideas presented in \cite{garavellohanpiccoli2016book,goatin2016speed} we describe a traffic flow network as a directed graph $G=(V,E)$, where the arcs $E$ correspond to roads and the vertices $V$ to junctions or intersections.

On each road $e\in E$ the density of cars is given by $\pe(x,t)$ and for a given initial state $\pe(x,0)\ \forall e\in E$ the dynamics are governed by conservation laws of the form 
\begin{align}\label{eq:network_model}
\partial_t \pe(t,x)+\partial_x f_e(t,x,\{\rho_k(t,\cdot)\}_{k\in E})=0\quad \forall e\in E, \, x\in (a_e,b_e),\, t>0.\end{align}
Equation \eqref{eq:network_model} allows for different choices of the flux functions $f_e$ as for example the well-studied \ac{LWR} model of local type
\begin{align}
    f_e(t,x,\{\rho_k(t,\cdot)\}_{k\in E})=\pe(t,x)v_e(\pe(t,x)),
\end{align}
where $v_e$ are suitable velocity functions.
On the other hand, the flux function can be of nonlocal type and hence depend on the whole density of the road $e$ at time $t$. 
In this work, we focus on the nonlocal case but first shortly give some remarks about the modeling of the local setting.

Generally, for solving hyperbolic \acp{PDE} on a network the boundary treatment is essential.
Therefore, the boundary (or coupling) conditions at junctions $v\in V$ have to be defined to ensure the conservation of mass.
The modeling of a junction in the case of the \ac{LWR} model is not unique and there exists a wide literature, see e.g. \cite{coclite2010vanishing,garavellohanpiccoli2016book,goatin2016speed}.
All these approaches describe the flow which passes directly through the junction depending on the purpose as for instance maximizing the flux through the junction \cite{goatin2016speed}, satisfying certain distribution rates \cite{garavellohanpiccoli2016book} or considering the vanishing viscosity approach \cite{coclite2010vanishing}.
Most of the models also share the similarity that they treat the junction only by determining the flows at the intersection point.
In contrast to that, for nonlocal fluxes in \eqref{eq:network_model} these boundary (or coupling) conditions are already present in a transition area in front of the junction point, coming from the nonlocal range.
This makes the network coupling of the nonlocal conservation laws more involved.

The model we are interested in this work is a nonlocal version of the \ac{LWR} model.
In the case of a single road the flux reads
\begin{align}\label{eq:nonlocalflux}
    f_1(t,x,\rho_1(t,\cdot))=\rho_1(t,x)V_1(t,x) \text{ with } V_1(t,x)=\int_x^{x+\ndt} v_1(\rho(t,y))\wt(y-x)dy
\end{align}
for any $\ndt>0.$
The model includes that drivers adapt their speed based on a weighted mean of  downstream velocities, where
$\ndt$ represents the nonlocal interaction range and $\wt$ is a kernel function.
In order to have a well-posed network model, we need the following assumptions.
\begin{assumption} \label{ass:vandkernel}
We impose the following assumptions on the velocity function $v_e$, $e\in E$, and the kernel function $\wt$:
\begin{equation} \label{eq:hypotheses}
\begin{aligned}
v_e \in C^2([0,\rho^{\max}_e];\mathbb{R}^+) & \colon v_e'\leq 0,\ v_e(\Rr)=0,\\
\wt \in C^1([0,\ndt];\mathbb{R}^+) & \colon \wt'\leq 0,\ \int_0^\ndt \wt(x) dx=1,\ \lim_{\ndt\to\infty}\wt(0)=0 \ \forall \ndt>0.
\end{aligned}
\end{equation}
\end{assumption}
The well-posedness under these assumptions on a single arc with $(a_1,b_1)=(-\infty,\infty)$ has been shown in \cite{friedrich2018godunov}.
Note that the considered kernel function $\wt$ is assumed to be equal for all roads $e \in E$ as it is rather a property of the driver than of the road.
In addition to the Assumptions \ref{ass:vandkernel}, we assume the following:
\begin{assumption}
We assume $\ndt <b_e-a_e\ \forall e\in E$.
\end{assumption}
This assumption ensures that the nonlocal range is restricted by the length $b_e-a_e$ of road $e$. Hence, we only consider one junction in the nonlocal downstream term. 

\subsection{Modeling of a junction}
In the following, we consider a set of incoming arcs $\mathcal{I}=\{1,\dots,M\}$ and a set of outgoing arcs $\mathcal{O}=\{M+1,\dots,M+N\}$ at a fixed junction.
In addition to simplify notation, we set $a_e=-\infty, \ \, b_e=0$ for $e=1,\dots,M$ and $a_e=0, \ \, b_e=\infty$ for $e=M+1,\dots,M+N$. 
As already noticed, coupling conditions are needed in order to be able to define a solution at the junction.
The coupling is induced by a function $\coupling_e$ and plays a crucial role in the junction modeling of network models governed by nonlocal conservation laws.
Inspired by equation \eqref{eq:nonlocalflux} we set
\begin{align}\label{eq:network_flux}
f_e(t,x,\{\rho_k\}_{k\in E})&= \begin{cases}\pe(t,x) V_e(t,x) +\coupling_e\left(\left\{\rho_i(t,\cdot)\right\}_{i\in\mathcal{I}}, \left\{ V_o(t,x)\right\}_{o\in\mathcal{O}}\right),&e\in\mathcal{I},\\[1ex]
 \pe(t,x)  V_e(t,x), &e\in\mathcal{O},
 \end{cases}
\end{align}
with 
\begin{align}\label{eq:velocity}
 V_e(t,x):=\begin{cases}
  \int_{\min\{x,0\}}^{\min\{x+\ndt,0\}} v_e(\pe(t,y)) \wt(y-x) dy,\quad &e\in\mathcal{I},\\[1ex]
 \int_{\max\{x,0\}}^{\max\{x+\ndt,0\}} v_e(\pe(t,y)) \wt(y-x) dy,\quad &e\in\mathcal{O}.
 \end{cases}
\end{align}
In addition, we couple equations \eqref{eq:network_model}, \eqref{eq:network_flux} with the initial conditions
\begin{align}\label{eq:initcond}
\rho_{e,0}(x)\in L^1 \cap \BV((a_e,b_e))\text{ for }x\in(a_e,b_e),\ e\in \{1,\dots,M+N\}.
\end{align}
This setting allows for different velocity functions and road capacities, respectively.
As can be seen the differences in the velocities are included by computing the mean velocity with the respective weights on the incoming or outgoing roads.
By having a closer look at the definition of $V_e(t,x)$ for $e\in \mathcal{O}$ in \eqref{eq:velocity} we recognize that the nonlocal velocities of the outgoing roads can become positive at $x=-\ndt$ as soon as drivers notice the junction and the properties of the next road.
The functional relationship on the coupling $\coupling_e$ do not only depend on the velocities of the outgoing roads but also on the densities of all incoming roads.
In particular, the changes in the maximum capacity $\rho^{\max}_e$ are captured by the function $\coupling_e$ such that no non-physical densities might occur.
Furthermore, we see that the velocities of the outgoing roads are important for all junction models due to the nonlocality of the problem, while the densities on the incoming roads play an important role at the junction for $M>1$.
Naturally, it can be reasonable to also derive models which include the velocities of all incoming roads or the densities of all outgoing roads.
Note that the latter are indirectly included through the nonlocal velocities.

Following \cite[Definition 4.2.4]{garavellohanpiccoli2016book} we define a weak solution for \eqref{eq:network_model} and \eqref{eq:network_flux} at a single junction as:
\begin{definition}[Weak solution]\label{def:weaksolnetwork}
A collection of functions $\pe\in C(\mathbb{R}^+;L^1_{loc}((a_e,b_e)))$ and $e\in\{1,\dots M+N\}$  with $a_e=-\infty, \ \, b_e=0$ for $e=1,\dots,M$ and $a_e=0, \ \, b_e=\infty$ for $e=M+1,\dots,M+N$, is called a weak solution of \eqref{eq:network_model} and \eqref{eq:network_flux} with \eqref{eq:velocity} if
\begin{enumerate}
\item for every $e \in \{1,\dots, M\}$ $\pe$ is a weak solution on $(-\infty,-\ndt)$ to
\begin{equation}\label{eq:NV}
\partial_t \pe(t,x) +\partial_x\left( \pe(t,x) \int_x^{x+\ndt} v_e(\pe(t,y))\wt (y-x)\dy\right)=0,
\end{equation}
and for every $e \in \{M+1,\dots,M+N\}$ on $(0,\infty)$ to \eqref{eq:NV};
\item for every $e \in \{1,\dots,M+N\}$ and for a.e. $t>0$, the function $x\mapsto \pe(t,x)$ has a bounded total variation;
\item for a.e. $t>0$, it holds 
\[\sum_{i=1}^M f_i(t,0-,\rho_i(t,0-))=\sum_{o=M+1}^{M+N} f_o(t,0+,\rho_o(t,0+));\]
\item the following integral equality holds
\begin{align*}
\sum_{i=1}^M \left( \int_{\mathbb{R}^+}\int_{-\infty}^0 \rho_i(t,x) \partial_t \testf_i(t,x)+ f(t,x,\rho_i(t,x)\partial_x \testf_i(t,x)dx dt\right)+\\
\sum_{o=M+1}^{M+N} \left( \int_{\mathbb{R}^+}\int_{0}^{+\infty} \rho_o(t,x) \partial_t \testf_o(t,x)+ f(t,x,\rho_o(t,x)\partial_x \testf_o(t,x)dx dt\right)=0
\end{align*}
for every collection of test function $\testf_i\in C_0^1((-\infty,0]\times \mathbb{R}_{\leq 0},\ i=1,\dots,M$ and $\testf_o\in C_0^1([0,\infty)\times \mathbb{R}_{\geq 0},\ o=M+1,\dots,M+N$ satisfying:
\[ \testf_i(0,\cdot)=\testf_o(0,\cdot),\qquad \partial_x\testf_i(0,\cdot)=\partial_x\testf_o(0,\cdot)\]
for $i\in \{1,\dots, M\}$ and $o\in\{M+1,\dots M+N\}$.
\end{enumerate}
\end{definition}

We note that in contrast to \cite[Definition 4.2.4]{garavellohanpiccoli2016book}, the functions $\pe$ have to be weak solutions {\em outside} the transition area and for $e \in \{M+1,\dots,M+N\}.$ 
This is ensured in the network setting by equation \eqref{eq:velocity}.
In addition to \cite[Definition 4.2.4]{garavellohanpiccoli2016book}, we introduce the last condition in Definition~\ref{def:weaksolnetwork} to also include the transition area into the definition of a weak solution.

Let us now define admissible conditions that weak solutions should satisfy.
Therefore, as in \cite{garavellohanpiccoli2016book}, we introduce the distribution matrix $A=(\alpha_{i,o})_{i=1,\dots,M,o=M+1,\dots, M+N}$ with $\sum_{o=M+1}^{N+M} \alpha_{i,o}=1$. Here, $\alpha_{i,o}$ determines the proportion of traffic coming from road $i$ and going to road $o$.

\begin{definition}[Admissible weak solution]\label{def:admisolnetwork} A collection of functions $\pe \in$ $C(\mathbb{R}^+;L^1_{loc}((a_e,b_e)))$ and $e\in\{1,\dots M+N\}$, with $a_e=-\infty, \ \, b_e=0$ for $e=1,\dots,M$ and $a_e=0, \ \, b_e=\infty$ for $e=M+1,\dots,M+N$, is called an admissible weak solution to \eqref{eq:network_model} and \eqref{eq:network_flux} if it is a weak solution in the sense of Definition \ref{def:weaksolnetwork} and additionally satisfies at least one of the following conditions: 
\begin{enumerate}[label=(\roman*)]
\item\label{item:admis1} If $\rho_i(t,x)\leq \rho_i^{\max}\, \forall i\in\{1,\dots,M\}$ a.e. $t>0$ and $x\in [-\ndt,0)$, then  for every $i \in \{1,\dots, M\}$ $\rho_i$ is an (entropy) weak solution on $(-\infty,0)$ to
\begin{equation*}
\partial_t \rho_i(t,x) +\partial_x\left( \rho_i(t,x) (V_i(t,x)+\sum_{o=M+1}^{N+M} \alpha_{i,o} V_o(t,x))\right)=0.
\end{equation*}
\item\label{item:admis2} For all $o\in\{M+1,\dots,M+N\}$ we have
\[f_o(t,0+,\rho_o(t,0+))=\sum_{i=1}^M \alpha_{i,o}f_i(t,0-,\rho_i(t,0-)).\]
\end{enumerate} 
\end{definition}
The first admissible condition introduced in Definition \ref{def:admisolnetwork} means the following: If no capacity restrictions of the outgoing roads affect the incoming road, the nonlocal model should produce the natural idea of the model \eqref{eq:nonlocalflux}.
This means that the drivers should adapt their velocity according to the mean velocity.
Here, the mean velocity is given by a part of the velocity on the current road and by a part of all outgoing roads weighted with the distribution parameters.
The second condition represents the distribution parameters which should be fulfilled at the junction.
As we will see, there exist couplings that do not always satisfy both conditions at the same time.
Hence, we have a kind of a trade off between satisfying the distribution parameters and the natural behavior of the model. 

The goal is now to find appropriate coupling functions $\coupling_e$ to close the model proposed in \eqref{eq:network_flux}.
We restrict to $M$-to-1 and 1-to-$N$ junctions since the coupling of $M$-to-$N$ junctions is more involved in the nonlocal setting. 
\begin{assumption}\label{assumptions}
We impose the following restrictions on the function $\coupling_e,\ e\in\mathcal{I}$ and the inflow on the outgoing roads $f_e(t,0),\ e\in\mathcal{O}$:
\begin{enumerate}
\item The inflow has to be positive and smaller than the maximum possible flow on the outgoing road, i.e.,
\begin{align}\label{eq:cond_inflow}
0\leq f_e(t,0) \leq \Rr V_e(t,0) \ \forall\ e\in\mathcal{O}.
\end{align}
\item The coupling function $\coupling_e$ has to be positive and smaller than the desired flow, i.e., the current density times the mean velocity of the outgoing roads:
\begin{align}
0 \leq \coupling_e\left(\left\{\rho_i(t,\cdot)\right\}_{i\in\mathcal{I}}, \left\{ V_o(t,x)\right\}_{o\in\mathcal{O}}\right)\leq \rho_e(t,x)\sum_{o=M+1}^{N+M} \alpha_{i,o} V_o(t,x).
\end{align}
\item The coupling function $\coupling_e$ does not decrease in the velocities of the outgoing roads and concerning the density of the road $e$ it is always smaller or equal than the corresponding maximum density, i.e.,
\begin{equation}\label{eq:cond_gdiff}
\begin{aligned}
\coupling_e(\pe(t,x),\dots)&\leq \coupling_e(\Rr,\dots)&\\
\coupling_e(\dots,V_k(t,x),\dots)&\leq \coupling_e(\dots,C,\dots)&\text{ for }k\in\mathcal{O},\ \forall C\ \geq V_k(t,x).
\end{aligned}
\end{equation}
\item The function $\coupling_e$ is Lipschitz continuous in $\pe$ with Lipschitz constant $L\leq \max_{e\in E} \| v_e \|_\infty$.
\end{enumerate}
\end{assumption}
The first two assumptions are limitations on the flows while  the third assumption states upper bounds on $\coupling_e$. 
Note that the second assumption together with the definition of the velocity \eqref{eq:velocity} implies that the first item of Definition \ref{def:weaksolnetwork} is satisfied for $e \in \{1,\dots, M\}$.

\subsection{Junction types}
In the following, we consider 1-to-1, 1-to-2 and 2-to-1 junctions, specify the assumptions from above and give concrete examples for the junction models.
The extensions to 1-to-N and M-to-1 junctions are then straightforward.
We note that by construction the proposed models satisfy 
the Rankine-Hugoniot condition, i.e., the third item of Definition \ref{def:weaksolnetwork}.

\subsubsection{1-to-1 junctions}
The 1-to-1 junction has been already studied extensively in \cite{chiarelloFriedrichGoatinGK}.
It has a kind of special role as it can be interpreted as a model on a single arc with the velocity function changing at the intersection point.
This consideration is not possible for the other types of junctions.
The coupling in \cite{chiarelloFriedrichGoatinGK} is given by
\begin{align} \label{eq:g}
\coupling_1(\rho_1(t,x),V_2(t,x))&:=\min\{\rho_1(t,x), \Rt\}V_2(t,x).
\end{align}
Under the Assumptions \ref{ass:vandkernel} the well-posedness and uniqueness of weak entropy solutions has been shown in \cite{chiarelloFriedrichGoatinGK}.
As we have just one incoming and one outgoing road, both items in the Definition \ref{def:admisolnetwork} are fulfilled.

\subsubsection{1-to-2 junctions}
For 1-to-2 junctions we have $M=1$ and $N=2$. So $\coupling_1$ only depends on $\po(t,x),\ V_2(t,x)$ and $V_3(t,x)$.
We need to prescribe distribution parameters $\alpha_{1,2}$ and $\alpha_{1,3}$ with $\alpha_{1,2}+\alpha_{1,3}=1$ which give us the desired flow from the incoming road to the outgoing roads.
These distribution rates can be for example determined through historical data.
Using the distribution rates we can specify the second condition of the Assumption \ref{assumptions}, i.e.,
\begin{align}\label{eq:cond_2for1to2}
0\leq \coupling_1(\po(t,x), V_2(t,x), V_3(t,x))\leq \po(t,x)\left(\alpha_{1,2}V_2(t,x)+\alpha_{1,3}V_3(t,x)\right).
\end{align}
The modeling choice of the function $\coupling_1$ is of course not unique. The function should preserve the densities in the given intervals (which is achieved by fulfilling Assumption \ref{assumptions}) and follow the purpose of modeling. Here we present two approaches: the first approach allows for a maximum flow (satisfying only \ref{item:admis1} in Definition~\ref{def:admisolnetwork}) and the second approach satisfies the distribution parameters at all costs (satisfying only  \ref{item:admis2} in Definition~\ref{def:admisolnetwork}).

\begin{example}\label{ex:12:maxflow}
The approach for the maximum flow is very similar to the 1-to-1 junction and is inspired by the model presented in \cite{goatin2016speed}.
Due to the distribution parameters the flow, from the incoming to one outgoing road is either given by the distribution rate times the flow or restricted by the maximum flow/capacity on the corresponding outgoing road. So we get
\begin{align}
\coupling_1(\po,V_2,V_3)=\min\{\alpha_{1,2}\po(t,x),\Rt\}V_2(t,x)+\min\{\alpha_{1,3}\po(t,x),\Rtr\}V_3(t,x).
\end{align} 
In the case that the capacity restrictions are not active we would simply have the flow defined by the density times mean velocity weighted by the distribution rates. The corresponding inflows on the outgoing roads are then given by:
\begin{align*}
f_e(t,0+,\po)=\min\{\alpha_{1,e}\po(t,0-),\Rr\}V_e(t,x),\quad e\in\{2,3\}.
\end{align*}
In addition, it is obvious that condition \ref{item:admis2} Definition~\ref{def:admisolnetwork} cannot be satisfied in all cases.
\end{example}

\begin{example}\label{ex:12:dist}
In order to always satisfy the distribution parameters, we adapt the idea introduced in \cite{garavellohanpiccoli2016book}.
The flow of the incoming road in the transition area is, if possible, the density times the mean (in terms of distribution rates) nonlocal velocity or the maximum feasible flows of the outgoing roads divided by the corresponding distribution, i.e.,
\begin{align}
\coupling_1(\po,V_2,V_3)=\min\left\{\po(t,x)(\alpha_{1,2}V_2(t,x)+\alpha_{1,3}V_3(t,x)),\frac{\Rt V_2(t,x)}{\alpha_{1,2}},\frac{\Rtr V_3(t,x)}{\alpha_{1,3}}\right\}
\end{align}
with the inflows
 \begin{align*}
f_e(t,0+,\pe)=\alpha_{1,e}f_1(t,0-,\po),\quad e\in\{2,3\}.
\end{align*}
If the desired flow is not restricted by the outgoing roads, also condition \ref{item:admis1} Definition~\ref{def:admisolnetwork} is satisfied.
\end{example}

\subsubsection{2-to-1 junctions}
Similar to the discussion above we intend to present again two models satisfying the Assumption \ref{assumptions}: one approach allows for the maximum possible flux and the other one satisfies the priority rules at all costs.
Therefore, priority rules have to be prescribed in the sense that the percentage of cars going from the incoming roads to the outgoing road is $q_{1,3}+q_{2,3}=1$.
In a first attempt, we assume the functional relationship
$\coupling_e(\pe(t,x),\rho_{e^-}(t,y(x)),V_3(t,x))$ with $e^-$ being the other incoming road and without specifying the density to be taken.
However, to keep a maximum principle, we also need to assume
\begin{align}\label{eq:cond_2to1}
\rho_{e^-}(t,y(x))=\rho_{e^-}(t,0-) \ \forall x\in(-\infty,0).
\end{align}
The inflow on the outgoing road is simply given by the sum of the two outflows at $x=0$, i.e.,
\begin{align*}
f_3(t,0+,\po,\pt)=\coupling_1(\po,\pt,V_3)+g_2(\pt,\po,V3).
\end{align*}
Therefore, both presented approaches satisfy \ref{item:admis2} Definition~\ref{def:admisolnetwork}.
\begin{example}\label{ex:21:maxflow}
To maximize the flux through the junction, several possibilities can be considered. 
First, the flux can pass on to the outgoing road without violating the flux restrictions coming from the maximum possible flux and the priority parameter.
Second, the flux restriction can become active. However, if there is not enough mass coming from the other incoming road we allow the flow to be higher allowing the maximum possible flow.
As in \cite{goatin2016speed}, this results in the following coupling function in the transition area:
\begin{align}
\coupling_e(\pe,\rho_{e^-},V_3)=\min\{\pe(t,x),\max\{ q_{e,3} \Rtr, \Rtr -\rho_{e^-}(t,0-)\}\}V_3(t,x).
\end{align}
\end{example}
\begin{example}\label{ex:21:prio}
To always satisfy the priority rules, we have to assume that $\pe(t,0)>0$ for all $t>0$ and $e\in\{1,2\}$.
Otherwise, the model does not give meaningful results since the solution is always to let no flow through the junction.
Inspired by \cite{garavellohanpiccoli2016book}, we want to maximize the flux through the junction but at the same time always satisfy the priority parameters.
Applying this idea to the transition area leads to
\begin{align}
 \coupling_e(\pe,\rho_{e^-},V_3)=\min\{\pe(t,x),q_{e,3} \Rtr, (q_{e,3}/q_{e^-,3})\rho_{e^-}(t,0-)\}V_3(t,x).   
\end{align}
\end{example}

\section{Numerical scheme}
So far we have presented junction models which seem to be a reasonable choice concerning the Definitions \ref{def:weaksolnetwork} and \ref{def:admisolnetwork}.
In this section, we now deal with the question of existence for weak solutions on networks.
To do so, we present a numerical discretization scheme of upwind type and consider its convergence properties.
We follow the ideas presented in \cite{chiarelloFriedrichGoatinGK} and \cite{friedrich2018godunov} to derive a scheme for \eqref{eq:network_model} and \eqref{eq:network_flux}.

The numerical scheme uses the following ingredients: For $j\in \Z, \ n\in \N$ and $e\in E$, let $x_{e,j-1/2}= j \Dx$ be the cell interfaces, $x_{e,j}=(j+1/2) \Dx$ the cells centers, corresponding to a  space  step  $\Dx $ such  that $\ndt=\Ne \Dx$ for  some $\Ne\in \N$, and let $t^n=n\Dt$ be the time mesh.
In particular, $x=x_{e,-1/2}=0$ is a cell interface.
Note that we assume the same step sizes for each road and that we have $j\geq 0$ iff $e\in\mathcal{O}$ and $j <0$ iff $e\in\mathcal{I}$.

The finite volume approximate solution is given by $\dr_e$ such that $\dr_e(t,x)=\rho_{e,j}^n$ for $(t,x)\in [t^n,t^{n+1}[\, \times  [x_{j-1/2},x_{j+1/2}[$ and $e\in E$. 
The initial datum $\rho_{e,0}$ in \eqref{eq:initcond} is approximated by the cell averages 
\begin{equation*}
\rho^{0}_{e,j}=\frac{1}{\Dx} \int_{x_{j-1/2}}^{x_{j+1/2}} \rho_{e,0}(x) dx, \qquad j\in\Z,\ e\in E.
\end{equation*}
Following \cite{chiarelloFriedrichGoatinGK,friedrich2018godunov}, we consider the numerical flux function
\begin{subequations}\label{eq:scheme}
\begin{equation}\label{eq:flux} 
F^n_{e,j} := 
\begin{cases}
\rho^n_{e,j} V^{n}_{e,j}+\coupling_e\left(\rho^n_{e,j},\{\rho^n_{i,-1}\}_{i\in\mathcal{I}},\{V^{n}_{o,j}\}_{o\in\mathcal{O}}\right), & e\in\mathcal{I}\\ 
\rho^n_{e,j} V^{n}_{e,j}, & e\in\mathcal{O},
\end{cases}
\end{equation}
with
    \begin{align}
 V^{n}_{e,j}&=
\begin{cases} 
 \sum_{k=0}^{\min\{-j-2,N_\eta-1\}} \gamma_k v_e(\rho^n_{e,j+k+1}), & e\in\mathcal{I},\\
\sum_{k=\max\{-j-1,0\}}^{N_\eta-1} \gamma_k v_e(\rho^n_{e,j+k+1}),& e\in\mathcal{O},
\end{cases}
 \label{eq:discretevelocity}\\
        \gamma_k&=\int_{k\Dx}^{(k+1)\Dx}\wt (x) dx, \qquad k=0,\ldots,N_{\eta}-1,\label{eq:kernel}
   \end{align}
where we set, with some abuse of notation, $\sum_{k=a}^b=0$ whenever $b<a$.
The influxes $F^n_{e,-1}$ for $e\in\mathcal{O}$ are defined by discrete versions of the fluxes given by the modeling approach.

In this way, we can define the following finite volume numerical scheme
    \begin{equation}\label{eq:evolution}
        \rho^{n+1}_{e,j}=\rho^n_{e,j}-\lambda \left(F^n_{e,j}-F^n_{e,j-1}\right)\text{ with }\lambda:= \frac{\Dt}{\Dx}.
    \end{equation}
\end{subequations}
Note that, due to the accurate calculation of the integral in \eqref{eq:kernel} and the definition of the convoluted velocities in \eqref{eq:discretevelocity}, there holds
\[0\leq V_j^{e,n} \leq v_e^{\max}\quad \forall\ j\in\mathbb{Z},\ n\in\mathbb{N},\ e \in E.\]
We set 
\[\Vert v \Vert:=\max_{e\in E}\Vert v_e\Vert_\infty, \quad \Vert v' \Vert:=\max_{e\in E}\Vert v'_e\Vert_\infty, \quad \Vert \rho \Vert:= \max_{e\in E}\Rr\] 
and consider the following CFL condition:
\begin{equation}\label{eq:CFL}
    \lambda \leq \frac{1}{\gamma_0 \Vert v'\Vert \Vert \rho \Vert +2\Vert v \Vert}.
\end{equation}
The proposed discretization for appropriate choices of $\coupling_e$ and the influxes are the basis to prove a maximum principle and the existence of weak solutions.

\subsection{Existence of weak solutions}
We prove that the scheme \eqref{eq:scheme} satisfies a maximum principle under the Assumption \ref{assumptions}.
We start by deriving some elementary inequalities concerning the differences of the velocities.
\begin{lemma}\label{lem:estimates}
Consider the velocities computed in \eqref{eq:discretevelocity} and let the conditions \eqref{eq:hypotheses} hold, then we have the following estimates
\begin{align}\label{eq:firstineq}
V_{e,j-1}^n-V_{e,j}^n&\leq\begin{cases}\begin{cases}
\gamma_0 \| v'\| (\Rr -\rho_{e,j}^n), &j\leq -1,\\
0, &j\geq 0,
\end{cases}& e\in\mathcal{I},\\
\begin{cases}
0, &j\leq -1,\\
\gamma_0 \| v'\| (\Rr -\rho_{e,j}^n), &j\geq 0,
\end{cases}& e\in\mathcal{O}.
\end{cases}
\end{align}
In addition, we have
\begin{align}\label{eq:secondineq}
V_{e,j-1}^n\Rr-V_{e,j}^n \rho_{e,j}^n&\leq\begin{cases}\begin{cases}
(\gamma_0 \| v'\| \|\rho\| +V_{e,j}^n)(\Rr -\rho_{e,j}^n), &j\leq -1,\\
V_{e,j}^n(\Rr -\rho_{e,j}^n), &j\geq 0,
\end{cases}& e\in\mathcal{I},\\
\begin{cases}
V_{e,j}^n(\Rr -\rho_{e,j}^n), &j\leq -1,\\
(\gamma_0 \| v'\| \|\rho\| +V_{e,j}^n) (\Rr -\rho_{e,j}^n), &j\geq 0,
\end{cases}& e\in\mathcal{O}.
\end{cases}
\end{align}
\end{lemma}
\begin{proof}
Let us consider the case $e\in\mathcal{I}$:
\begin{align*}
    V_{e,j}^{n}-V_{e,j}^{n}&=
    \begin{cases}
    \begin{aligned}
    &\sum_{k=1}^{\Ne-1}(\gamma_k-\gamma_{k-1})v_e(\rho_{e,j+k}^n)-\gamma_{\Ne-1}v_e(\rho_{e,j+\Ne}^n)\\
    &+\gamma_0 v_e(\rho_{e,j}^n),
    \end{aligned}
    \qquad &j\leq -\Ne-1,\\
    \sum_{k=1}^{-j-1}(\gamma_k-\gamma_{k-1})v_e(\rho_{e,j+k}^n)+\gamma_0 v_e(\rho_{e,j}^n),\qquad &-\Ne\leq j \leq -2,\\
    \gamma_0 v_e(\rho_{e,-1}^n),\qquad & j=-1,\\
    0,\qquad & j \geq 0,\\
    \end{cases}\\
    &\leq \begin{cases}
    \gamma_0 v_e(\rho_{e,-1}^n),\qquad & j\leq-1,\\
    0,\qquad & j \geq 0.\\
    \end{cases}
\end{align*}
Using $v_e(\Rr)=0$ and the mean value theorem yields inequality \eqref{eq:firstineq} for $e\in\mathcal{I}$. Then, the inequality \eqref{eq:secondineq} follows from multiplying the first equation by $\Rr$ and adding and subtracting $V_{e,j}^n\rho_{e,j}^n$.

The inequalities for $e\in\mathcal{O}$ can be obtained analogously.
\end{proof}
Next, we give the details of the proof in the particular case of 1-to-2 situations.
\begin{proposition}
Under hypothesis~\eqref{eq:hypotheses} and the CFL condition \eqref{eq:CFL}, the sequence generated by the numerical scheme~\eqref{eq:scheme} for a 1-to-2 junction with distribution parameters $\alpha_{1,2}+\alpha_{1,3}=1$ satisfies the following maximum principle:
\begin{align*}
    0\leq \rho_{1,j}^n\leq \Ro \quad \text{for}\quad j\leq -1\quad\text{ and }\quad 0\leq \rho_{e,j}^n\leq \Rr \quad \text{for}\quad j \geq 0,\quad e\in\{2,3\} \quad \forall n\in \mathbb{N}.
\end{align*}
\end{proposition}
\begin{proof}
We start with the incoming road and the lower bound which can be obtained by applying \eqref{eq:cond_2for1to2}:
\begin{align*}
\rho_{1,j}^{n+1}&=\rho_{1,j}^n+\lambda \left(F_{1,j-1}^n-F_{1,j}^n\right)\\
&\geq \rho_{1,j}^n- \lambda \rho_{1,j}^n V_{1,j}^n - \lambda \coupling_1(\rho_{1,j}^n, V_{2,j}^n, V_{3,j}^n)\\
&\geq \rho_{1,j}^n (1-\lambda \left( V_{1,j}^n +\alpha_{1,2} V_{2,j}^n+ \alpha_{1,3} V_{3,j}^n\right)\\
&\geq 0.
\end{align*}
To obtain the upper bound we use the third and fourth property on $g_e$ of Assumption \ref{assumptions} and the estimates from Lemma \ref{lem:estimates}:
\begin{align} \nonumber
&\rho_{1,j}^{n+1}\\ \nonumber
&\leq \rho_{1,j}^n + \lambda \left( V_{1,j-1}^n \Ro -V_{1,j}^n \rho_{1,j}^n\right) +\lambda \left( \coupling_1(\Ro,V_{2,j-1}^n,V_{3,j-1}^n)- g_1(\rho_{1,j}^n, V_{2,j}^n, V_{3,j}^n) \right)\\ \label{eq:step2to1}
&\leq \rho_{1,j}^n + \lambda \left( \gamma_0 \| v'\|\| \rho \|+ V_{1,j}^n\right)(\Ro- \rho_{1,j}^n) +\lambda \left( \coupling_1(\Ro,V_{2,j}^n,V_{3,j}^n)- g_1(\rho_{1,j}^n, V_{2,j}^n, V_{3,j}^n) \right)\\ \label{eq:stepCFLdiff}
&\leq \rho_{1,j}^n + \lambda \left( \gamma_0 \| v'\|\| \rho \|+ V_{1,j}^n +L\right)(\Ro- \rho_{1,j}^n)\\ \nonumber
&\leq \rho_{1,j}^n + \lambda \left( \gamma_0 \| v'\|\| \rho \|+ V_{1,j}^n 2\| v\|\right)(\Ro- \rho_{1,j}^n)\\ \nonumber
&\leq \Ro.
\end{align}
For the outgoing roads, we have to consider the cell $j=0$ since for $j>0$ the maximum principle is given in~\cite{friedrich2018godunov}. Here, we use the conditions on the inflow which is positive and obtain with $e \in\{2,3\}$:
\begin{align*}
\rho_{e,0}^{n+1} \geq \rho_{e,0}^n-\lambda F_{e,0}
\geq \rho_{e,0}^n (\lambda-\| v\|)
\geq 0.
\end{align*}
Similarly with Lemma \ref{lem:estimates},
\begin{align*}
\rho_{e,0}^{n+1}&=\rho_{e,0}^n+\lambda(F_{e,-1}^n-F_{e,0}^n)\\
&\leq \rho_{e,0}^n+\lambda(\Rr V_{e,-1}^n-\rho_{e,0}^n V_{e,0}^n)\\
& \leq \rho_{e,0}^n+\lambda(\gamma_0\| v'\| \| \rho\| +\| v\|)(\Rr-\rho_{e,0}^n)\\
&\leq\Rr.\qedhere
\end{align*}
\end{proof} 
\begin{remark}
The proof for the 2-to-1 junction is completely analogous except that we additionally need to satisfy condition \eqref{eq:cond_2to1}.
This is necessary when proving the upper bound on the incoming roads.
Instead of \eqref{eq:step2to1} we have, e.g., for road $1$:
\[ \rho_{1,j}^{n+1} \leq \rho_{1,j}^n + \lambda \left( \gamma_0 \| v'\|\| \rho \|+ V_{1,j}^n\right)(\Ro- \rho_{1,j}^n) +\lambda \left( \coupling_1(\Ro,\rho_{2,-1}^n,V_{3,j}^n)- g_1(\rho_{1,j}^n,\rho_{2,-1}^n, V_{3,j}^n) \right).\]
As we assume \eqref{eq:cond_2to1} we can proceed as in \eqref{eq:stepCFLdiff} and apply the Lipschitz continuity of $\coupling_e$.
The rest of the proof remains the same.
\end{remark}
\begin{remark}
Note that the CFL condition \eqref{eq:CFL} can be further relaxed.
This is possible due to the properties of the kernel function $\wt$ and the definition of the velocities in \eqref{eq:velocity} and \eqref{eq:discretevelocity}, respectively.
Exemplary, for the 1-to-2 junction we can estimate $V_{1,j}^n+L$ in \eqref{eq:stepCFLdiff} by $\norm{v}$.
Hence, the relaxed CFL condition reads
\begin{align*}
    \lambda \leq \frac{1}{\gamma_0 \Vert v'\Vert \Vert \rho \Vert +\Vert v \Vert}.
\end{align*}
\end{remark}
In order to prove a BV bound in space we consider the following total variation:

\[\sum_{i\in\mathcal{I}} \sum_{j<-1} \abs{\rho_{i,j+1}^n-\rho_{i,j}^n}+
\underbrace{\sum_{o\in\mathcal{O}} \abs{\rho_{o,1}^n-\rho_{o,0}^n}}_{\leq 2N\norm{\rho}}+
\underbrace{\sum_{o\in\mathcal{O}}\sum_{i\in\mathcal{I}} \abs{\rho_{o,0}^n-\rho_{i,-1}}}_{\leq 2NM\norm{\rho}}+
\sum_{o\in\mathcal{O}} \sum_{j>0} \abs{\rho_{o,j+1}^n-\rho_{o,j}^n}.\]

As already displayed, two terms can be easily estimated by the total number of (incoming) roads and the norm of $\rho$.
For the other two terms it is possible to derive an estimate which includes an exponential constant multiplied by the total variation of the initial condition.
Hence, following \cite[Theorem 3.2]{friedrich2018godunov}, we are able to derive an upper bound for the last term. However,
to estimate the first term, we need to use a regularization of the limiter as all proposed couplings $\coupling_e$ are only weakly differentiable in $\rho_e$.
A suitable regularization of the minimum function is for instance given in \cite[eq. (4.9)]{chiarelloFriedrichGoatinGK}.
This enables us to follow the proof of \cite[Lemma 2]{chiarelloFriedrichGoatinGK} and bound the first term.

 Taking the maximum gives an estimate on the total variation of the form
\[ TV(\rho)\leq \exp(C(T))TV(\rho_0)+\text{const}.\]
Applying \cite[Theorem 3.3]{friedrich2018godunov} and similar steps as already described above we get a BV estimate in space and time such that with Helly's Theorem the convergence of a sub-sequence can be concluded.
Without mentioning all the details we aim to demonstrate that the collection of limiting functions $\rho_e^*$ are weak solutions in the sense of the Definition \ref{def:weaksolnetwork} and \ref{def:admisolnetwork}.
The first item of Definition \ref{def:weaksolnetwork} can be shown by using Lax-Wendroff types arguments, the calculations done in \cite{chiarelloFriedrichGoatinGK,friedrich2018godunov} as well as equations \eqref{eq:velocity} and \eqref{eq:discretevelocity}, respectively.
The calculation can be even simplified since we deal with weak solutions.
The second condition in Definition \ref{def:weaksolnetwork} holds by the BV estimates derived above and Helly's Theorem, the third condition by the definition of the coupling conditions.
The fourth condition can be also obtained by using Lax-Wendroff type arguments.

Now we turn to the conditions \ref{item:admis1} and \ref{item:admis2} in Definition \ref{def:admisolnetwork}.
Starting with \ref{item:admis1}, we see that the Examples \ref{ex:12:maxflow}--\ref{ex:21:prio} satisfy
\[\coupling_e(\{\rho_i(t,x)\}_{i\in\mathcal{I}},\{V_o(t,x)\}_{o\in\mathcal{O}})\leq\pe(t,x) \sum_{o=M+1}^{N+M} \alpha_{i,o} V_o(t,x)\]
if $\rho_i(t,x)\leq \rho_o^{\max}, \; \forall o\in\mathcal{O},\ i\in\mathcal{I}$.
Equality is only obtained in the Example \ref{ex:12:maxflow}, while in the other models further conditions have to be satisfied.
In contrast, the condition \ref{item:admis2} is satisfied for all models by construction, except Example \ref{ex:12:maxflow}.

Therefore, we can conclude that weak solutions in the sense of the Definitions \ref{def:weaksolnetwork} and \ref{def:admisolnetwork} exist.
Note that we do not consider the uniqueness of those solutions, since the standard techniques are not applicable in a straightforward way.

\section{Limit \texorpdfstring{$\ndt \to \infty$}{}}
As mentioned in the introduction nonlocal traffic flow models have been introduced to incorporate the challenges occurring in nowadays traffic such as autonomous cars.
The nonlocal range $\ndt$ can be therefore interpreted as a connection radius between cars, where the latter only need information about the downstream traffic. 
Apparently, in case of non-autonomous cars, the nonlocal range can be seen as the sight of a human driver.
The question arises what happens if autonomous cars would have perfect information about the downstream traffic. 
To treat this question from a theoretical point of view,
the nonlocal range should tend to infinity in the network setting  \eqref{eq:network_model} and \eqref{eq:network_flux}. 
For a similar traffic model on a single road, this feature has been already analyzed in \cite[Corollary 1.2]{ChiarelloGoatin}.
Therein, the model tends to a linear transport equation with maximum velocity.
From the modeling perspective this result is kind of intuitive:
All drivers know exactly what happens in front of them on the whole road such that they can react in advance.
So they are able to keep the speed at the maximum velocity regardless of the current traffic situation.
However, a network model includes different types of junctions and different maximum densities or speed functions, respectively, leading to a non-intuitive behavior of drivers. 

For our considerations, we  start with some general estimates on the nonlocal velocities and discuss the 1-to-1 junction model in more detail (since all necessary estimates have been already established in \cite{chiarelloFriedrichGoatinGK}).
Note that we now use the notation $V_e^\ndt(t,x)$ for the velocities to emphasize their dependence on the nonlocal range $\ndt$.

\begin{lemma}\label{lem:velinf}
Let $\Omega_\mathcal{I}=(0,\infty)\times(-\infty,0)$ and $\Omega_\mathcal{O}=(0,\infty)\times(0,\infty)$. Define $\Omega_e=\Omega_\mathcal{I}$ if $e \in \mathcal{I}$ and $\Omega_e=\Omega_\mathcal{O}$ if $e \in \mathcal{O}$. In addition, let $\rho_e\in L^\infty( \Omega_e)$ be a collection of functions which are either of compact support on $\Omega_e$ or identical to zero and let $K_e$ be a compact subset of $\Omega_e$. Then, we obtain for the velocities \eqref{eq:velocity} the following:
\begin{equation}\label{eq:LDCVeta}
\begin{aligned}
&\lim_{\ndt\to\infty} \iint_{K_e} |V^\ndt_e(t,x)|=0,\quad e\in\mathcal{I},\qquad \lim_{\ndt\to\infty} \iint_{K_e} |V^\ndt_e(t,x)-v_e(0)|=0,\quad e\in\mathcal{O},\\
&\lim_{\ndt\to\infty} \iint_{K_i} |V^\ndt_e(t,x)-v_e(0)|=0,\quad e\in\mathcal{O}, i\in\mathcal{I}.
\end{aligned}
\end{equation}
\end{lemma}

\begin{proof}
We start with $e\in\mathcal{I}$.
If $\rho_e(t,\cdot)\equiv 0$ we directly obtain 
\begin{align*}
V_e^\ndt(t,x)= \int_x^{\min\{x+\ndt,0\}} \wt(y-x)v_e(0)dy\leq \wt(0) v_e(0) \abs{\min\{x+\ndt,0\}-x\}}
\end{align*}
which goes to zero for $\ndt \to\infty$ due to \eqref{eq:hypotheses}.
Now let us consider the case $\rho_e(t,\cdot)$ being of compact support.
Without loss of generality, we assume that the support is on $[a_e(t),b_e(t)]$ with $a_e(t)<b_e(t)\leq 0$.
Here, we have
\begin{align*}
V_e^\ndt(t,x)=&\int_{\max\{x,a_e(t)\}}^{\min\{b_e(t),x+\ndt\}} \wt (y-x) v_e(\pe(t,y))dy+\int_{\min\{x,a_e(t)\}}^{\min\{a_e(t),x+\ndt\}} \wt (y-x) v_e(0)dy+\\
&\int_{\max\{x,b_e(t)\}}^{\min\{\max\{x+\ndt,b_e(t)\},0\}} \wt (y-x) v_e(0)dy\\
\stackrel{\ndt\to\infty}{\to}&\int_{\max\{x,a_e(t)\}}^{b_e(t)} \wt (y-x) v_e(\pe(t,y))dy+\int_{\min\{x,a_e(t)\}}^{a_e(t)} \wt (y-x) v_e(0)dy+\\
&\int_{\max\{x,b_e(t)\}}^{0} \wt (y-x) v_e(0)dy\\
\leq& \wt(0) v_e(0)\underbrace{\left(\abs{b_e(t)-\max\{x,a_e(t)\}}+\abs{a_e(t)-\min\{x,a_e(t)\}}+\abs{\max\{x,b_e(t)\}}\right)}_{<\infty}\\
\stackrel{\ndt\to\infty}{\to} &0.
\end{align*} 
As we consider the pointwise limit, $t$ and $x$ are fixed and so are $a_e(t)$ and $b_e(t)$. 
Hence, the intervals in the last estimate are all finite and with $\lim_{\ndt\to\infty}\wt(0)=0$, cf. \eqref{eq:hypotheses}, we obtain that the upper bound on $V_e^\ndt(t,x)$ goes to zero.
As  we also have $V_e^\ndt(t,x)\geq 0$, we can conclude that $V_e^\ndt(t,x)\to 0$ pointwise for $\ndt \to \infty$ on $\Omega_{\mathcal{I}}$.

If we consider $e\in\mathcal{O}$, $x>0$ and $0\leq a_e(t)<b_e(t)$, we get
\begin{align*}
V_e^\ndt(t,x)=&\int_{\max\{x,a_e(t)\}}^{\min\{b_e(t),x+\ndt\}} \wt (y-x) v_e(\pe(t,y))dy+
\int_{\min\{x,a_e(t)\}}^{\min\{a_e(t),x+\ndt\}} \wt (y-x) v_e(0)dy+\\
&\int_{\max\{x,b_e(t)\}}^{\max\{x+\ndt,b_e(t)\}} \wt (y-x) v_e(0)dy\\
\stackrel{\ndt\to\infty}{=}&\int_{\max\{x,a_e(t)\}}^{b_e(t)} \wt (y-x) v_e(\pe(t,y))dy+
\int_{\min\{x,a_e(t)\}}^{a_e(t)} \wt (y-x) v_e(0)dy+\\
&\int_{\max\{x,b_e(t)\}}^{\infty} \wt (y-x) v_e(0)dy\\
\to&\ v_e(0).
\end{align*}
Since the first two terms can be again estimated from above by
\[\wt(0)v_e(0) \left(\abs{\max\{x,a_e(t)\}-b_e(t)}+\abs{\min\{x,a_e(t)\}-a_e(t)}\right)\to 0\]
and for the last term we can use that $\int_0^\infty \wt(y) dy=1$ due to \eqref{eq:hypotheses}, we end up with
\begin{align*}
1\geq \int_{\max\{x,b(t)\}}^{\infty} \wt (y-x)dy= 1-\int_x^{\max\{x,b(t)\}} \wt(y-x)dy\geq 1-\wt(0)\abs{b(t)-x}\to 1
\end{align*}
for $\ndt \to \infty$.
If $\rho_e(t,\cdot)\equiv 0$, we proceed as for the last term.
Hence, we have $V_e^\ndt(t,x)\to v_e(0)$ pointwise for $e\in\mathcal{O}$ on $\Omega_\mathcal{O}$.
In addition, if we consider $x\in(-\ndt,0)$ the above calculations for $V_e^\ndt(t,x)$ and $e\in\mathcal{O}$ are completely similar by setting $x=0$ in the lower bounds of the integrals.
Note that the case $x\leq -\ndt$ disappears in the pointwise limit by choosing $\ndt$ large enough for a fixed $x$.
So we also obtain $V_e^\ndt(t,x)\to v_e(0)$ pointwise for $e\in\mathcal{O}$ on $\Omega_\mathcal{I}$.

In addition, we have $V_e^\ndt(t,x)\leq v_e(0)\ \forall x,\ t,\ e\in E $ with $v_e(0)$ being an integrable function on each compact subset of $\Omega_e$.
Hence, Lebesgue's dominated convergence theorem yields the assertion.
\end{proof}

Next, we prove the convergence for $\ndt\to\infty$ for the 1-to-1 junction model. 
Note that for simplicity we write $\rho$ instead of $\rho_1$ and $\rho_2$, since in the case of a 1-to-1 junction the solution is uniquely given on whole $\mathbb{R}$.

\begin{proposition}\label{prop:1to1inf}
Let the hypotheses \eqref{eq:hypotheses} hold and let $\rho_0\in\textbf{BV}(\mathbb{R},I)$.
Then, the solution $\rho^\ndt$ of the 1-to-1 junction model given by \eqref{eq:network_model}, \eqref{eq:network_flux} with \eqref{eq:g} converges for $\ndt\to\infty$ to the unique entropy solution of the local problem
\begin{equation}\label{eq:ADRmodel}
\begin{aligned}
\partial_t \rho +\partial_x \left(\min\{\rho v_2(0),\Rt v_2(0)\}\right)&=0\\
\rho(0,x)&=\rho_0(x).
\end{aligned}
\end{equation}
\end{proposition}

\begin{proof}
We first note that the 1-to-1 model can be rewritten by defining $\rho(t,x)=\rho_1(t,x)$ if $x<0$ and $\rho(t,x)=\rho_2(t,x)$ if $x>0$ as
\[\partial \rho(t,x)+\partial_x \left(\rho(t,x)V_1(t,x)+g(\rho(t,x),V_2(t,x))\right)=0,\]
which simplifies the notation in the rest of the proof, see also \cite{chiarelloFriedrichGoatinGK}.

The existence of solutions for $\ndt \to \infty$ is given since the BV estimates and the maximum principle in \cite{chiarelloFriedrichGoatinGK} are uniform as $\ndt\to\infty$ such that Helly's Theorem yields up to a subsequence the convergence of the solution in the $L^1_{loc}$ norm.

We start from the entropy inequality and add and subtract $v_2(0)$ at the right place, such that for a fixed  $\kappa\in\mathbb{R}$ and $\phi\in C_0^1([0,\infty)\times \mathbb{R};\mathbb{R}^+)$, we obtain
\begin{align*}
0\leq&  \int_0^\infty \int_{-\infty}^\infty (|\rho-\kappa|\phi_t +|\rho-\kappa| V_1\phi_x +\sgn(\rho-\kappa)(\tilde \coupling(\rho)-\tilde \coupling(\kappa))(V_2\pm v_2(0))\phi_x\\
&-\sgn(\rho-\kappa)\kappa \partial_x V_1\phi-\sgn(\rho-\kappa)\tilde \coupling(\kappa) \partial_x V_2\phi)(t,x)dxdt+\int_{-\infty}^{\infty} |\rho-\kappa|\phi(x,0)dx,
\end{align*}
where $\tilde \coupling$ is defined as $\tilde \coupling(\rho)=\min\{\rho,\Rt\}$. 
First, note that due to the compact support of the initial condition and the finite speed of the waves, which are bounded by $\| v\|$, the solution $\rho(t,x)$ is also of compact support such that we can apply Lemma \ref{lem:velinf} in the following.
Additionally, as the test functions are of compact support, there exist $T>0$ and $R>0$ such that $\phi(t,x)=0$ for $\abs{x}>R$ or $t>T$.
Using the latter and Lemma \ref{lem:velinf}, we obtain for $\ndt\to\infty$:
\begin{align*}
\int_0^\infty \int_{-\infty}^\infty|\rho-\kappa| V_1\phi_x dxdt &\leq \left(\Vert \rho\Vert+\abs{\kappa}\right) \Vert \phi_x\Vert \int_0^T \int_{-R}^R|V_1| dx dt \to 0\\
\int_0^\infty \int_{-\infty}^\infty \sgn(\rho-\kappa)(\tilde \coupling(\rho)-\tilde \coupling(\kappa))(V_2- v_2(0))\phi_x dxdt &\leq 2\Vert \rho\Vert \Vert \phi_x\Vert \int_0^T \int_{-R}^R|V_2-v_2(0)| dx dt \\
&\to 0.
\end{align*}
Again, thanks to the compactness of the test function and  condition \eqref{eq:hypotheses} of the kernel, we have
\begin{align*}
\int_0^\infty \int_{-\infty}^\infty \sgn(\rho-\kappa)\kappa \partial_x V_1\phi dxdt &\leq 16\kappa\Vert \phi\Vert\|v\| T R \wt(0) \to 0,\\
\int_0^\infty \int_{-\infty}^\infty \sgn(\rho-\kappa)\tilde \coupling(\kappa) \partial_x V_2\phi) dxdt &\leq 16\norm{\rho}\kappa\Vert \phi\Vert\|v\| T R \wt(0) \to 0.
\end{align*}
Here, we use that
\begin{align*}
\left| \partial_x V_1\right|
\leq& \left| \int_{\min\{x,0\}}^{\min\{x+\ndt,0\}} v_1(\rho(t,y))\wt'(y-x)dy\right|+v_1(\rho(t,\min\{x,0\}))\wt(0)\\
&+v_1(\rho(t,\min\{x+\ndt,0\}))\wt(0)\\
\leq& \|v\| \left|\int_{\min\{x,0\}}^{\min\{x+\ndt,0\}} \wt'(y-x)dy\right| + 2 \| v\| \wt(0)\\
\leq& 4 \| v\| \wt(0),
\end{align*}
and analogously we get $\left| \partial_x V_2\right|\leq 4 \| v\| \wt(0)$.
We are left with 
\begin{align*}
0\leq&  \int_0^\infty \int_{-\infty}^\infty (|\rho-\kappa|\phi_t +\sgn(\rho-\kappa)(g(\rho)-g(\kappa))(v_2(0))\phi_x)(t,x)dxdt+\int_{-\infty}^{\infty} |\rho-\kappa|\phi(x,0)dx,
\end{align*}
which is the entropy inequality of the corresponding local model \eqref{eq:ADRmodel}.
\end{proof}

Some remarks are in order.
\begin{remark}
The proof can be easily adapted to the model without  considering any junctions, i.e.,
\begin{align*}
\partial_t \rho +\partial_x \left(V^\ndt(t,x)\rho\right)&=0\\
\rho(0,x)&=\rho_0(x),
\end{align*}
with $V^\ndt(t,x)=\int_x^{x+\ndt} v(\rho(t,y))\wt(y-x)dy$.
As already mentioned for a similar model a convergence result for $\ndt \to \infty$ has already been discovered in \cite[Corollary 1.2]{ChiarelloGoatin}. Using the proof above we obtain the convergence to the unique weak entropy inequality of the linear transport equation:
\begin{align*}
\partial_t \rho +\partial_x \left(v(0)\rho\right)&=0\\
\rho(0,x)&=\rho_0(x),
\end{align*} 
i.e.,
\[0\leq \int_0^\infty \int_{-\infty}^\infty (|\rho-\kappa|\phi_t +|\rho-\kappa| v(0)\phi_x dx dt +\int_{-\infty}^{\infty} |\rho-\kappa|\phi(x,0)dx.\]
The linear transport is a situation on the road where no traffic jams occur since even if the traffic is at maximum density all cars drive at the fasted velocity possible.
Even though this sounds more like an idealistic property, the model captures the limiting behavior for an infinite interaction range resulting in an ``optimized'' traffic.
\end{remark}  

\begin{remark}
The model \eqref{eq:ADRmodel} has been also recovered in \cite{ArmbrusterDegondRinghofer2006} in the context of production with velocity $v_2(0)$ and capacity $v_2(0)\Rt$.
Transferring the idea of a production model to traffic flow, we observe that all cars want to move at a constant speed $v_2(0)$ as long as there is enough capacity, which is mainly determined by the maximum density of the second road $\Rt$.

Let us explain the model dynamics with the help of Riemann initial data, i.e.,
\[\rho_0(x)=\begin{cases}
\rho_L,&\text{ if }x<0,\\
\rho_R,&\text{ if }x>0,
\end{cases}
\]
with $\rho_L\in[0,\Ro]$ and $\rho_R\in[0,\Rt]$.
We recall that the change in the velocity is located at $x=0$.
There are only two possible solutions:
\begin{itemize}
\item if $\rho_L\leq\Rt$, the solution is given by a linear transport with velocity $v_2(0)$,
\item if $\rho_L>\Rt$, we have a rarefaction wave with the density equal to $\Rt$ as an intermediate state.
\end{itemize}
Both cases are illustrated in Figure \ref{fig:RP}.

\setlength{\fwidth}{0.8\textwidth}
\begin{figure}[ht]
\centering
\begin{tikzpicture}

\begin{axis}[%
width=0.4\fwidth,
height=0.4\fwidth,
at={(0\fwidth,0\fwidth)},
scale only axis,
xmin=-1,
xmax=1,
xtick={0},
xticklabels={$0$},
xlabel style={font=\color{white!15!black}},
xlabel={$x$},
ymin=0.4,
ymax=1.1,
ytick={0.5,0.75,1},
yticklabels={$0.5$,$0.75$,$1$},
ylabel style={font=\color{white!15!black}},
ylabel={$\rho$},
axis background/.style={fill=white},
title style={font=\bfseries},
title={$\Rt=1$},
legend style={at={(0.9,1.1)}, anchor=south west, legend cell align=left, align=left, draw=white!15!black}
]
\addplot [color=black, line width=1.5pt]
  table[row sep=crcr]{%
-1	1\\
0	1\\
0	0.5\\
1	0.5\\
};
\addlegendentry{$\rho(0,x)$}

\addplot [color=black, dotted, line width=1.5pt]
  table[row sep=crcr]{%
-1	1\\
0.5	1\\
0.5	0.5\\
1	0.5\\
};
\addlegendentry{$\rho(t,x)\ t>0$}

\end{axis}

\begin{axis}[%
width=0.4\fwidth,
height=0.4\fwidth,
at={(0.55\fwidth,0\fwidth)},
scale only axis,
xmin=-1,
xmax=1,
xtick={0},
xticklabels={$0$},
xlabel style={font=\color{white!15!black}},
xlabel={$x$},
ymin=0.4,
ymax=1.1,
ytick={0.5,0.75,1},
yticklabels={$0.5$,$0.75$,$1$},
ylabel style={font=\color{white!15!black}},
ylabel={$\rho$},
title style={font=\bfseries},
title={$\Rt=0.75$},
axis background/.style={fill=white}
]
\addplot [color=black, line width=1.5pt]
  table[row sep=crcr]{%
-1	1\\
0	1\\
0	0.5\\
1	0.5\\
};

\addplot [color=black, dotted, line width=1.5pt]
  table[row sep=crcr]{%
-1	1\\
0	1\\
0 0.75\\
0.25 0.75\\
0.25	0.5\\
1	0.5\\
};

\end{axis}

\end{tikzpicture}
\caption{Solution of the two Riemann problems for \eqref{eq:ADRmodel}.}\label{fig:RP}
\end{figure}
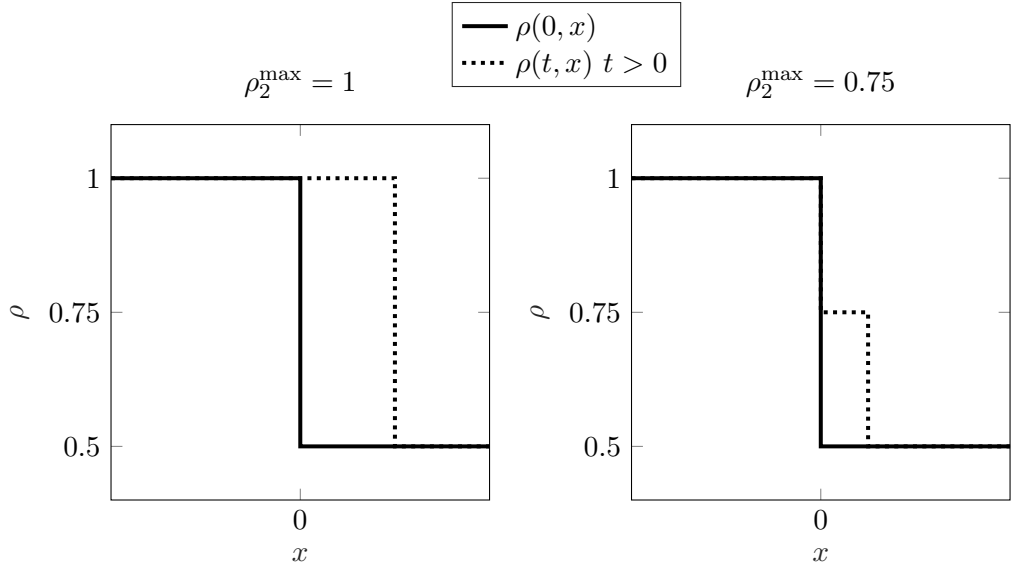

In the first case, i.e., the linear transport, no traffic jams occurs and traffic is transported at a constant velocity.
As we allow the drivers to accelerate as soon as they are aware of the junction, which they always are in the limit $\ndt \to\infty$, they do not respect the given maximum velocity on the first road.
They always drive with the maximum velocity of the second road which might be higher or lower than the one of the first road.
So the model ``optimizes''  the traffic but might not care about velocity restrictions.
The Riemann problems producing a rarefaction wave induce a congestion at the end of the first road and beginning of the second road due to the fact that the initial traffic on the first road cannot pass completely onto the second road. 
Again, the model solves this problem in a reasonable way in terms of traffic congestion, i.e., as much flow as possible is sent through the junction at a constant speed.
\end{remark}
Let us now turn to the more general nonlocal network setting.
Here, we need to satisfy the local version of Definition \ref{def:weaksolnetwork}, i.e., \cite[Definition 4.2.4]{garavellohanpiccoli2016book}.
As aforementioned, these are basically the items 1-3 of Definition \ref{def:weaksolnetwork}, even now $\rho_e$ for $e\in\{1,\dots,M\}$ has to be a weak solution on $(-\infty,0)$ instead of $(-\infty,-\ndt)$. Obviously, the Rankine-Hugoniot condition is also satisfied for $\ndt \to \infty$.
Since the BV estimates established in \cite[Theorem 3.2]{friedrich2018godunov} and \cite[Lemma 2]{chiarelloFriedrichGoatinGK} also hold for $\ndt \to \infty$ thanks to $\lim_{\ndt\to\infty}\wt(0)=0$, the second item is clear.
The calculations are similar to above.
In particular, for the outgoing roads in the examples \ref{ex:12:maxflow}--\ref{ex:21:prio} the dynamics in the limit are described by
\[\partial_t \pe +\partial_x\left( \pe v_e(0) \right)=0.\]
For the incoming roads the convergence of the coupling $\coupling_e$ plays the most important role.
By obtaining the pointwise limit of $\coupling_e$ we can proceed similarly as in the proof of Proposition \ref{prop:1to1inf} also using Lemma \ref{lem:velinf}. Therefore, we obtain the following limit models for $x<0$:
\begin{itemize}
\item Example \ref{ex:12:maxflow}:
\[\partial_t \po +\partial_x\left(\min\{\alpha_{1,2}\po(t,x),\Rt\}v_2(0)+\min\{\alpha_{1,3}\po(t,x),\Rtr\}v_3(0)\right)=0,\]
\item Example \ref{ex:12:dist}:
\[\partial_t \po +\partial_x\left(\min\{\po(t,x)(\alpha_{1,2}v_2(0)+\alpha_{1,3}v_3(0)),\frac{\Rt v_2(0)}{\alpha_{1,2}},\frac{\Rtr V_3(0)}{\alpha_{1,3}}\}\right)=0,\]
\item Example \ref{ex:21:maxflow}, $e\in\{1,2\}$:
\[\partial_t \pe +\partial_x\left(\min\{\pe(t,x),\max\{q_{e,3}\Rtr,\Rtr-\rho_{e^-}(t,0-)\}\}v_3(0)\right)=0,\]
\item Example \ref{ex:21:prio}, $e\in\{1,2\}$:
\[\partial_t \pe +\partial_x \left(\min\{\pe(t,x),q_{e,3}\Rtr,q_{e,3}/q_{e^-,3}\rho_{e^-}(t,0-)\}v_3(0)\right)=0.\]
\end{itemize}
By construction all models are similar to a production type model and inherit the property that the flow moving at a constant speed is restricted by a capacity.

We note that the computations for the Example \ref{ex:12:dist} are slightly more involved since the nonlocal velocities are inside the minimum.

\section{Numerical simulations}
In the following we consider a more complex network to demonstrate the properties of the proposed models.
The network under consideration has a diamond structure and consists of nine roads and six vertices, see Figure \ref{fig:diamond}.

\setlength{\fwidth}{0.8\textwidth}
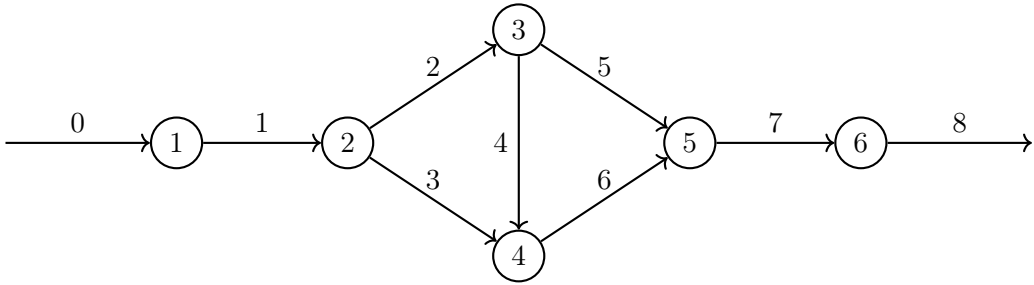
\begin{figure}[ht]
\centering
\begin{tikzpicture}[scale=1.5]

\node[circle,draw, thick]  (1) at (0, 0) {1};
\node[circle,draw, thick]  (2) at (1.5, 0) {2};
\node[circle,draw, thick]  (3) at (3, 1) {3};
\node[circle,draw, thick]  (4) at (3, -1) {4};
\node[circle,draw, thick]  (5) at (4.5,0) {5};
\node[circle,draw, thick]  (6) at (6,0) {6};
\draw [->,thick] (-1.5,0) -- (1) node[midway, above]{0};
\draw [->,thick] (1) -- (2) node[midway, above]{1};
\draw [->,thick] (2) -- (3) node[midway, above]{2};
\draw [->,thick] (2) -- (4) node[midway, above]{3};
\draw [->,thick] (6) -- (7.5,0) node[midway, above]{8};
\draw [->,thick] (3) -- (4) node[midway, left]{4};
\draw [->,thick] (3) -- (5) node[midway, above]{5};
\draw [->,thick] (4) -- (5) node[midway, above]{6};
\draw [->,thick] (5) -- (6) node[midway, above]{7};

\end{tikzpicture}
\caption{Network structure}\label{fig:diamond}
\end{figure}
As we have not prescribed inflow and outflow conditions in this work the roads $0$ and $8$ are used as artificial roads to avoid posing boundary conditions at intersections.
In particular, due to the nonlocality of the dynamics, we need to impose more than just a single value boundary condition at the end of the network.
The length of these artificial roads is set to $(a_0,b_0)=(-\infty,0)$ and $(a_8,b_8)=(0,\infty)$, respectively.
For all other roads $e=1,\dots,7$ we set $(a_e,b_e)=(0,1)$.
We also remark that we do not consider the artificial roads when calculating traffic measures, i.e., the outflow of the system is measured at the end of road 7.
At the vertices $2$ and $3$ we have to prescribe distribution parameters and priority parameters at the vertices $4$ and $5$.
These are given by $\alpha_{1,2}=\alpha_{1,3}=0.5,\ \alpha_{2,4}=1/5,\ \alpha_{2,5}=4/5$ and $q_{3,6}=4/5,\ q_{4,6}=1/5,\ q_{5,7}=4/5,\ q_{6,7}=1/5$.
The velocity function on all roads $e\in\{0,\dots,8\}$ is described by
\begin{equation}\label{eq:velnetworkexample}
    v_e(\rho)=v_e^{\max}(1-\rho)
\end{equation}
and hence the maximum density on all roads is equal.
In addition, the initial conditions are given by
\[\rho_{e,0}(x)=\bar \rho_e\ \forall x\in (a_e,b_e),\ e\in\{0,\dots,8\}\]
and explicit parameters can be found in Table \ref{tab:diamondparas}.

\begin{table}[ht]
    \caption{Parameters of the diamond network}
    \label{tab:diamondparas}
    \centering
    \begin{tabular}{c| c c c c c c c c c}
        Road $e$ & 0 & 1 & 2 & 3 & 4 & 5 & 6 & 7& 8\\ \hline
        $\bar \rho_e$ &0.4 & 0.4 & 0.4 & 0.4 & 0.8& 0.4&0.8&0.2&0.2\\
        $v_e^{\max}$&0.5 &  0.5 & 2&2 &0.5&2&0.5&1&1
    \end{tabular}
\end{table}
\noindent We note that the delicate point in this network is the choice of the maximum velocities and initial densities for the roads 4 and 5.
Due to the high velocity on road 5 and the shorter distance to the end of road 7, the way over the more congested road 4 seems to be not favorable in case of measuring travel times. In all simulations, we choose a linear decreasing kernel function $\wt=2(\ndt-x)/\ndt^2$ and use $\Dx=0.01$ combined with an adaptive CFL condition determined by \eqref{eq:CFL} for all norms in each time step $t^n$.
In the following, we denote by the \emph{nonlocal maximum flux model} the network model using the coupling condition described in Example \ref{ex:12:maxflow} for the 1-to-2 junctions at the vertices 2 and 3 and the one described in Example \ref{ex:21:maxflow} for the 2-to-1 junctions at the vertices 4 and 5. 
Analogously, the \emph{nonlocal distribution model} corresponds to the coupling conditions of Example \ref{ex:12:dist} and Example \ref{ex:21:prio}.

For the comparison of the different network models we consider the following traffic measures, see e.~g. \cite{goatin2016speed, reilly2015adjoint,treiber2013traffic}:
\begin{itemize}
    \item total travel time
    \begin{align*}
        TTT=\sum_{e=1}^7 \int_0^T\int_0^{1} \rho_e(t,x) dx dt,
    \end{align*}
    \item outflow
    \begin{align*}
        O= \int_0^T f_7(t, 1, \rho_7(t,1) )dt,
    \end{align*}
    \item congestion measure
    \begin{align*}
        CM=\sum_{e=1}^7 \int_0^T\max\left\lbrace 0, \int_0^{1} \rho_e(t,x)- \frac{f_e(t, x, \rho_e(t,x) )}{v_{e,\text{ref}}} dx\right\rbrace dt.
    \end{align*}
\end{itemize}
With some abuse of notation the flux can be either local or nonlocal depending on the models used. We also remark that the reference velocity is chosen to be
$v_{e,\text{ref}}=0.5v_e^{\max}.$

\subsection{Nonlocal models}
We start by comparing the \emph{nonlocal maximum flux} with the \emph{nonlocal distribution model}.
Therefore, we set $\ndt=0.5$ and consider the final time of $T=20$.
Figure \ref{fig:nonlocalsol} displays the approximate solution at the final time $T$.

\setlength{\fwidth}{0.8\textwidth}
\begin{figure}[h!]
\centering
\input{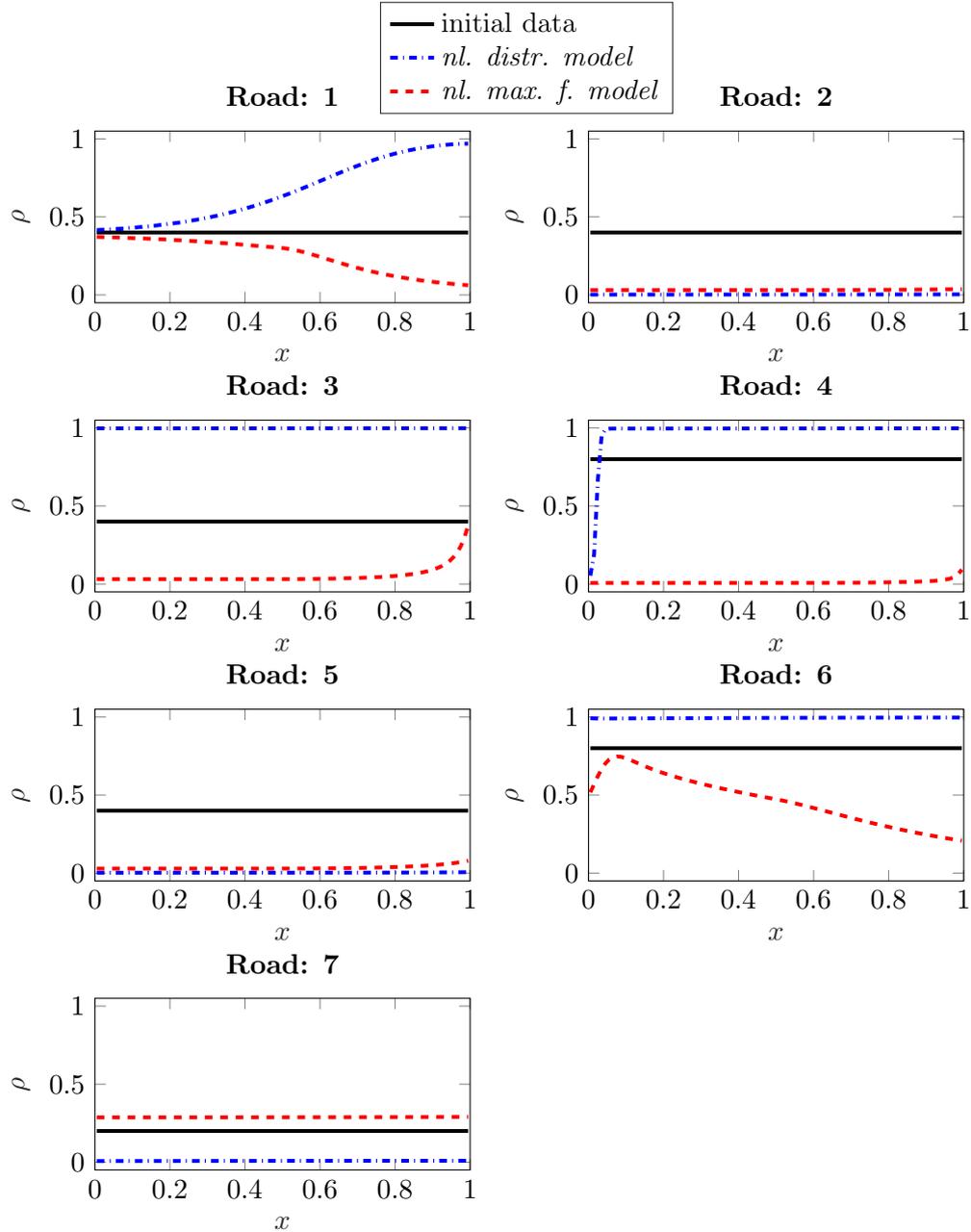}
\caption{Approximate solutions of the \emph{nonlocal maximum flux} (\emph{nl. max. f. model}) and the \emph{nonlocal distribution model} (\emph{nl. dist. model}) at $T=20$ for different roads and $\ndt=0.5$.}\label{fig:nonlocalsol}
\end{figure}

Apparently, the \emph{nonlocal distribution model} results in a congested network.
Over time road 6 becomes even further congested with the traffic jam moving backwards to road 4 and also to road 1 over road 3 while the other roads are rather empty.

  \begin{table}[ht]
        \caption{Traffic measures}
      \label{tab:nonlocalmodels}
      \centering
      \begin{tabular}{c|c c c}
          model & Outflow   & $TTT$   &   congestion  \\ \hline
             \emph{distribution} &    2.1531&     59.696&      48.744\\  
    \emph{maximum flux}&   4.6774&     36.011&      16.144
      \end{tabular}
  \end{table}

These effects can be also recognized in Table \ref{tab:nonlocalmodels} where the outflow of the \emph{nonlocal maximum flux model} is more than twice as high as in the congested \emph{nonlocal distribution model}. Similar observations can be made for the total travel time ($TTT$) and the congestion measure. 
The main reason for this behavior is that in the \emph{nonlocal distribution model} the prescribed distribution rates must be exactly kept.
On the contrary, the \emph{nonlocal maximum flux} does not necessarily fulfill them exactly.
In particular, the actual distribution\footnote{Actual means that the inflows of the outgoing roads over time are divided by the outflow of the incoming road.} at the vertex 3 over time shows that the distribution onto road $5$ is in the interval $[0.93,0.98]$ instead of the prescribed value of $0.8$.
As a consequence less vehicles enter road 4 and a traffic jam can be avoided.
Furthermore, in the \emph{nonlocal maximum flux model} the actual priority parameters at vertex $5$ start to change away from the prescribed values around approximately $t=4$ such that for $t\in [5,20]$ the ratio from road $6$ to $7$ is higher as from road $5$ to $7$.
This helps to resolve the traffic jam at road $7$.
The \emph{nonlocal distribution model} model keeps again the prescribed parameters and is not able to resolve the traffic jam.

This example also demonstrates that both models are in line with the original modeling ideas:
The \emph{nonlocal distribution model} model obeys the distribution rates (by construction) while the \emph{nonlocal maximum flux model} achieves maximum fluxes.

\subsection{Nonlocal vs. local models}
Next, we compare the nonlocal modeling approaches to local network models.
These are described by \eqref{eq:network_model} with
\[f_e(t,x,\left\lbrace \rho_k\right\rbrace_{k\in E})=\rho_e(t,x)v_e\left(\rho_e(t,x)\right)\]
for $e\in E$, $t>0$ and $x\in (a_e,b_e)$.
The local model is equipped with coupling conditions satisfying demand and supply functions, cf. \cite{goatin2016speed}:
\begin{equation}\label{eq:supplyanddemand}
D_e(\rho)=\begin{cases}
\rho v_e(\rho), &\text{if }\rho\leq \sigma_e,\\
\sigma_e v_e(\sigma_e), &\text{if }\rho >\sigma_e,
\end{cases}
\qquad
S_e(\rho)=\begin{cases}
\sigma_e v_e(\sigma_e), &\text{if }\rho\leq \sigma_e,\\
\rho v_e(\rho), &\text{if }\rho >\sigma_e,
\end{cases}
\end{equation}
where $\sigma_e$ is the maximum point of the flux function $\rho v_e(\rho)$\footnote{We note that the choice of the velocity function in \eqref{eq:velnetworkexample} ensures the existence of a unique maximum point.}.
As already mentioned the coupling conditions for the \emph{nonlocal maximum flux model} are also inspired by this approach.
For completeness these are in the local model for the 1-to-2 junction (using the notation from the previous sections)
\begin{align*}
    f_e(t,0+,\rho_1,\rho_e)&=\min\{\alpha_{1,e}D_1(\rho_1(t,0-)),S_e(\rho_e(t,0+))\},\quad e\in\{2,3\}\\
    f_1(t,0-,\rho_1,\rho_2,\rho_3)&=f_2(t,0+,\rho_1,\rho_2)+f_3(t,0+,\rho_1,\rho_3)
\end{align*}
and for the 2-to-1 junction:
\begin{align*}
    &f_e(t,0-,\rho_e,\rho_{e^-})=\min\{D_e(\rho_{e}(t,0-)),\max\{q_{e,3} S_3(\rho_3(t,0+)),S_3(\rho_3(t,0+))-D_{e^-}(\rho_{e^-}(t,0-))\}\}\\
    &f_3(t,0+,\rho_1,\rho_2,\rho_3)=f_1(t,0-,\rho_1,\rho_2)+f_2(t,0-,\rho_2,\rho_1).
\end{align*}
In contrast the local coupling conditions inspiring the \emph{nonlocal distribution model} can be found in \cite{garavellohanpiccoli2016book} and can be summarized as:
for the 1-to-2 junction
\begin{align*}
    f_1(t,0-,\rho_1,\rho_2,\rho_3)&=\min\{D_1(\rho_1(t,0-)),S_2(\rho_2(t,0+))/\alpha_{1,2},S_3(\rho_3(t,0+))/\alpha_{1,3}\}\\
    f_e(t,0+,\rho_1,\rho_e)&=\alpha_{1,e}f_1(t,0-,\rho_1,\rho_2,\rho_3),\quad e\in\{2,3\}
\end{align*}
and for the 2-to-1 junction:
\begin{align*}
    f_e(t,0-,\rho_e,\rho_{e^-})&=\min\{D_e(\rho_{e}(t,0-)),q_{e,3}/q_{e^-,3}D_{e^-}(\rho_{e^-}(t,0-)),q_{e,3} S_3(\rho_3(t,0+))\}\\
    f_3(t,0+,\rho_1,\rho_2,\rho_3)&=f_1(t,0-,\rho_1,\rho_2)+f_2(t,0-,\rho_2,\rho_1).
\end{align*}
In order to solve the local models numerically we use the Godunov scheme with the coupling conditions as presented in \cite{goatin2016speed}.

\setlength{\fwidth}{0.5\textwidth}
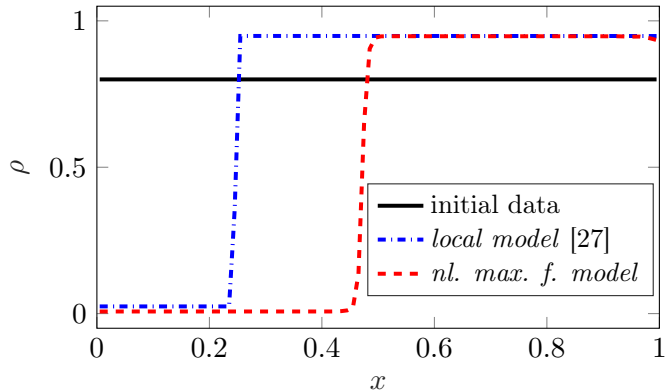
\begin{figure}[ht]
\centering
%
%
\begin{tikzpicture}

\begin{axis}[%
width=0.951\fwidth,
height=0.549\fwidth,
at={(0\fwidth,0\fwidth)},
scale only axis,
xmin=0,
xmax=1,
xlabel style={font=\color{white!15!black}},
xlabel={$x$},
ymin=-0.05,
ymax=1.05,
ylabel style={font=\color{white!15!black}},
ylabel={$\rho$},
axis background/.style={fill=white},
legend style={at={(0.99,0.45)},legend cell align=left, align=left, draw=white!15!black,font=\small}
]
\addplot [color=black, line width=1.5pt]
  table[row sep=crcr]{%
0.005	0.8\\
0.995	0.8\\
};
\addlegendentry{initial data}

\addplot [color=blue, dashdotted, line width=1.5pt]
  table[row sep=crcr]{%
0.005	0.0246054270398115\\
0.235	0.024605427039864\\
0.245	0.356831852049792\\
0.255	0.948437395831202\\
0.995	0.948191451112349\\
};
\addlegendentry{\emph{local model} \cite{goatin2016speed}}

\addplot [color=red, dashed, line width=1.5pt]
  table[row sep=crcr]{%
0.005	0.00726037347246067\\
0.415	0.00732439118354389\\
0.425	0.00763716249314161\\
0.435	0.00859742143420883\\
0.445	0.0110013340887527\\
0.455	0.0187472536944775\\
0.465	0.134384635431842\\
0.475	0.649016784948384\\
0.485	0.904674358481893\\
0.495	0.942703836760103\\
0.505	0.946976624510711\\
0.515	0.947438027333322\\
0.545	0.947486825038852\\
0.895	0.947317292208527\\
0.905	0.947221404443864\\
0.925	0.947251647016505\\
0.935	0.946836070856792\\
0.945	0.945833995409545\\
0.965	0.946305882595702\\
0.975	0.943832983745418\\
0.985	0.939624979812737\\
0.995	0.934492213823163\\
};
\addlegendentry{\emph{nl. max. f. model}}

\end{axis}

\end{tikzpicture}%
\caption{Approximate solution of the \emph{local model} with the coupling conditions from \cite{goatin2016speed} and the \emph{nonlocal maximum flux model} (\emph{nl. max. f. model}) at $T=20$ for $\ndt=0.05$ on road 4.}\label{fig:solutionlocal}
\end{figure}

Figure \ref{fig:solutionlocal} shows the approximate solutions of the local supply and demand approach and the \emph{nonlocal maximum flux model} at $T=20$ and for $\ndt=0.05$ on road $4$.
It can be seen that both models lead to different solutions at road $4$ as the end of the traffic jam is located further downstream in the nonlocal models.
Considering smaller values of $\ndt$ (and also $\Dx$) this behavior does not change.
Hence, this example provides numerical evidence that in the network case for the \emph{nonlocal maximum flux model} no convergence for $\ndt\to 0$ to the local network model can be expected.
We remark that in \cite{chiarelloFriedrichGoatinGK} it was observed that the simple 1-to-1 junction tends towards the vanishing viscosity solution.
This seems also not the case for the \emph{nonlocal maximum flux model}.
The vanishing viscosity solution of the considered network behaves very similar to the solution obtained by the supply and demand approach and hence convergence can be ruled out. 
\begin{remark}
Note that for the vanishing viscosity approach no distribution and priority parameters have to be prescribed and the maximum densities on all roads have to be equal. However, an approximate solution can be obtained by the numerical scheme presented in \cite{towers2020explicit}.
\end{remark}
Let us now compare the traffic measures for the \emph{nonlocal maximum flux model} with $\ndt\in \{0.5,0.25,0.1, 0.05\}$ and the corresponding local model.
Obviously, the nonlocal models perform better than the local model regarding the traffic measures while the advantages increase even further with larger nonlocal range $\ndt$.

  \begin{table}[ht]
        \caption{Traffic measures for the \emph{nonlocal maximum flux model} and the local model with the coupling conditions from \cite{goatin2016speed} at $T=20$.}
      \label{tab:local}
      \centering
      \begin{tabular}{c|c c c}
          model & Outflow   & $TTT$   &   congestion  \\ \hline
    $\ndt=0.5$ & 4.6774 &    36.011  &    16.144   \\
    $\ndt=0.25$& 4.3651  &   39.589   &   19.114\\
    $\ndt=0.1$&4.1546    & 42.432     & 21.611\\
    $\ndt=0.05$&4.0719   &  43.627    &  22.752 \\
    \cite{goatin2016speed} &3.7862   &  47.268     &  26.09   
      \end{tabular}
  \end{table}

Now, we compare the \emph{nonlocal distribution model} to the corresponding local model with coupling conditions from \cite{garavellohanpiccoli2016book}.
For small time and large $\ndt$ the nonlocal solution provides different dynamics as the local model.
Even though these effects are less for larger times and/or smaller values of $\ndt$ as displayed in Figure \ref{fig:differenttime}.
Here, we selected exemplary two roads, namely road 1 and road 4.
For small time periods we see that the solution with large $\ndt$ is different to the local solution but has a kind of  smoothing effect across the transition area.
For $\ndt$ small at $T=1$ and $T=20$, the nonlocal approximate solution suggests a numerical convergence to the local approximate solution.
It is interesting to see that for larger times and $\ndt$ the nonlocal effects become less significant.
In particular, road $1$ is the only road displaying visible differences between the local and nonlocal solution while the other roads behave similarly to road $4$.
Nevertheless, even on road $1$ both models result in a traffic jam but they differ how the transition from free flow to traffic jam is created.

\setlength{\fwidth}{0.8\textwidth}
\begin{figure}[ht]
\centering
\input{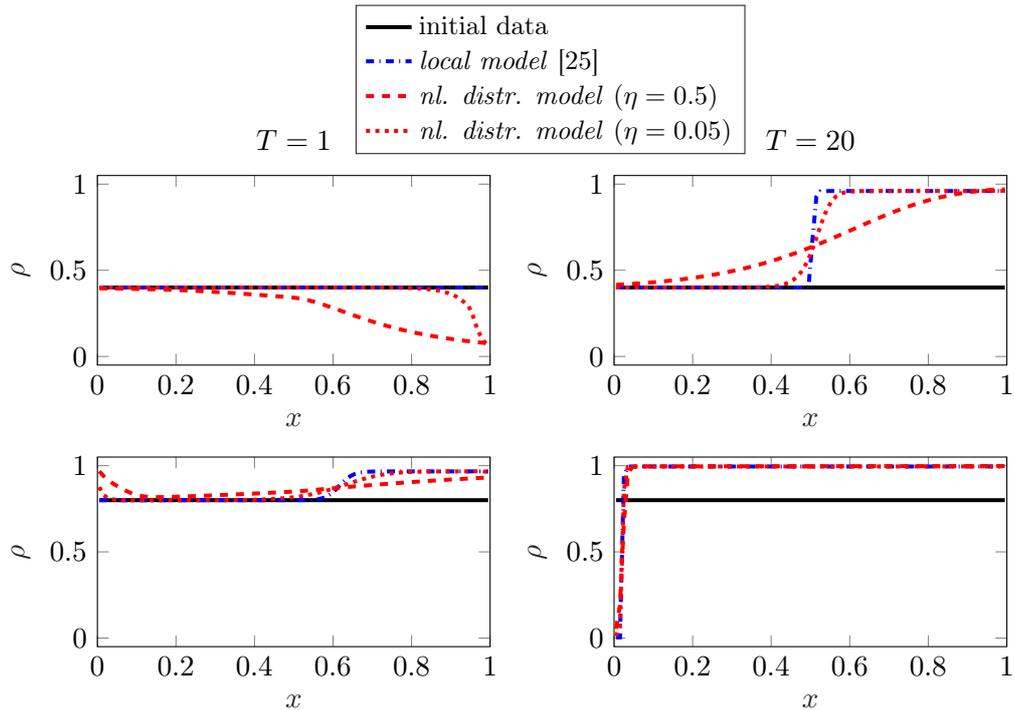}
\caption{Approximate solutions of the \emph{local model} with the coupling conditions from \cite{garavellohanpiccoli2016book} and the \emph{nonlocal distribution model} (\emph{nl. distr. model}) with different values of $\ndt$ at $T=1$ (left column) and $T=20$ (right column) for the roads 1 (top row) and 4 (bottom row).}\label{fig:differenttime}
\end{figure}

In addition, the traffic measures displayed in Table \ref{tab:localdist} computed at the final time $T=20$ support the observation that in this example the \emph{nonlocal distribution model} model behaves similarly as its local counterpart.

  \begin{table}[ht!]
  \caption{Traffic measures for the \emph{nonlocal distribution model} and the local model with the coupling conditions from \cite{garavellohanpiccoli2016book} at $T=20$.}
      \label{tab:localdist}
      \centering
      \begin{tabular}{c|c c c}
          model & Outflow   & $TTT$   &   congestion  \\ \hline
    $\ndt=0.5$ & 2.1531  &   59.696     & 48.744     \\
    $\ndt=0.25$& 2.1485  &   60.189     & 48.219\\
    $\ndt=0.1$&2.1455    & 60.665       &47.96\\
    $\ndt=0.05$&2.1446   &  60.858       & 47.9 \\
    \cite{garavellohanpiccoli2016book} &2.1434    & 61.192  &    47.782   
      \end{tabular}
  \end{table}

Nevertheless, we intend to stress that these observations are only due to the choice of the specific example.
We notice that in other scenarios the differences between the \emph{nonlocal distribution model} and its local counterpart are more significant for large times and $\ndt$, even though for smaller values of $\ndt$ the solution seems to numerically converge to the local one.
In our numerical study we do not find any example to rule out the convergence as we do for the \emph{nonlocal maximum flux model}.

\section*{Conclusion}
We have introduced a network model for nonlocal traffic. The modeling is essentially based on coupling conditions for 1-to-1, 2-to-1 and 1-to-2 junctions.
Using a finite volume numerical scheme we can show a maximum principle and the existence of weak solutions for the network model. Further investigations include the consideration of 
the limiting behaviour for $\ndt \to \infty$. 
The numerical simulations demonstrate the ideas of the proposed junction models. We also investigate 
the limit $\ndt \to 0$ numerically and notice that in case of the \emph{nonlocal maximum flux model} convergence to its local counterpart can be ruled out while 
in the \emph{nonlocal distribution model} convergence can be observed.

Future work will include the investigation of a general $M$-to-$N$ junction.
We also intend to derive coupling conditions for other nonlocal modeling equations, such as the second order traffic flow model proposed in \cite{friedrich2020micromacro}.

\section*{Acknowledgments} 
The financial support of the DFG project GO1920/10 is acknowledged.

\bibliographystyle{siam}
\bibliography{literatur}

\def\ocirc#1{\ifmmode\setbox0=\hbox{$#1$}\dimen0=\ht0 \advance\dimen0
  by1pt\rlap{\hbox to\wd0{\hss\raise\dimen0
  \hbox{\hskip.2em$\scriptscriptstyle\circ$}\hss}}#1\else {\accent"17 #1}\fi}
\begin{thebibliography}{10}

\bibitem{ArmbrusterDegondRinghofer2006}
{\sc D.~Armbruster, P.~Degond, and C.~Ringhofer}, {\em A model for the dynamics
  of large queuing networks and supply chains}, SIAM J. Appl. Math., 66 (2006),
  pp.~896--920.

\bibitem{aw2000resurrection}
{\sc A.~Aw and M.~Rascle}, {\em Resurrection of ``second order'' models of
  traffic flow}, SIAM J. Appl. Math., 60 (2000), pp.~916--938.

\bibitem{BayenKeimer-preprint}
{\sc A.~Bayen, A.~Keimer, L.~Pflug, and T.~Veeravalli}, {\em Modeling
  multi-lane traffic with moving obstacles by nonlocal balance laws}.
\newblock Preprint, 2020.

\bibitem{BlandinGoatin2016}
{\sc S.~Blandin and P.~Goatin}, {\em Well-posedness of a conservation law with
  non-local flux arising in traffic flow modeling}, Numer. Math., 132 (2016),
  pp.~217--241.

\bibitem{bressan2020entropy}
{\sc A.~Bressan and W.~Shen}, {\em Entropy admissibility of the limit solution
  for a nonlocal model of traffic flow}, arXiv preprint arXiv:2011.05430,
  (2020).

\bibitem{bressan2020traffic}
\leavevmode\vrule height 2pt depth -1.6pt width 23pt, {\em On traffic flow with
  nonlocal flux: a relaxation representation}, Archive for Rational Mechanics
  and Analysis, 237 (2020), pp.~1213--1236.

\bibitem{camilli2018measure}
{\sc F.~Camilli, R.~De~Maio, and A.~Tosin}, {\em Measure-valued solutions to
  nonlocal transport equations on networks}, J. Differential Equations, 264
  (2018), pp.~7213--7241.

\bibitem{chalons2018high}
{\sc C.~Chalons, P.~Goatin, and L.~M. Villada}, {\em High-order numerical
  schemes for one-dimensional nonlocal conservation laws}, SIAM Journal on
  Scientific Computing, 40 (2018), pp.~A288--A305.

\bibitem{chiarelloFriedrichGoatinGK}
{\sc F.~A. Chiarello, J.~Friedrich, P.~Goatin, S.~G{\"o}ttlich, and O.~Kolb},
  {\em {A non-local traffic flow model for 1-to-1 junctions}}, European Journal
  Appl. Math.,  (2019), p.~1–21.

\bibitem{friedrich2020micromacro}
{\sc F.~A. Chiarello, J.~Friedrich, P.~Goatin, and S.~Göttlich}, {\em
  Micro-macro limit of a nonlocal generalized aw-rascle type model}, SIAM
  Journal on Applied Mathematics, 80 (2020), pp.~1841--1861.

\bibitem{ChiarelloGoatin}
{\sc F.~A. Chiarello and P.~Goatin}, {\em Global entropy weak solutions for
  general non-local traffic flow models with anisotropic kernel}, ESAIM Math.
  Model. Numer. Anal., 52 (2018), pp.~163--180.

\bibitem{ChiarelloGoatinmulticlass}
{\sc F.~A. Chiarello and P.~Goatin}, {\em Non-local multi-class traffic flow
  models}, Netw. Heterog. Media, 14 (2019), pp.~371 -- 387.

\bibitem{shen2019stationary}
{\sc J.~Chien and W.~Shen}, {\em Stationary wave profiles for nonlocal particle
  models of traffic flow on rough roads}, NoDEA Nonlinear Differential
  Equations Appl., 26 (2019), p.~Paper No. 53.

\bibitem{forcadel2021nonlocal}
{\sc I.~Ciotir, R.~Fayad, N.~Forcadel, and A.~Tonnoir}, {\em A non-local
  macroscopic model for traffic flow}, ESAIM Math. Model. Numer. Anal., 55
  (2021), pp.~689--711.

\bibitem{coclite2020general}
{\sc G.~M. Coclite, J.-M. Coron, N.~De~Nitti, A.~Keimer, and L.~Pflug}, {\em A
  general result on the approximation of local conservation laws by nonlocal
  conservation laws: The singular limit problem for exponential kernels}, arXiv
  preprint arXiv:2012.13203,  (2020).

\bibitem{coclite2010vanishing}
{\sc G.~M. Coclite and M.~Garavello}, {\em Vanishing viscosity for traffic on
  networks}, SIAM Journal on Mathematical Analysis, 42 (2010), pp.~1761--1783.

\bibitem{coclite2005network}
{\sc G.~M. Coclite, M.~Garavello, and B.~Piccoli}, {\em Traffic flow on a road
  network}, SIAM J. Math. Anal., 36 (2005), pp.~1862--1886.

\bibitem{colombo2020local}
{\sc M.~Colombo, G.~Crippa, E.~Marconi, and L.~V. Spinolo}, {\em Local limit of
  nonlocal traffic models: convergence results and total variation blow-up}, in
  Annales de l'Institut Henri Poincar{\'e} C, Analyse non lin{\'e}aire,
  Elsevier, 2020.

\bibitem{colombo2006multilane}
{\sc R.~M. Colombo and A.~Corli}, {\em Well posedness for multilane traffic
  models}, Ann. Univ. Ferrara Sez. VII Sci. Mat., 52 (2006), pp.~291--301.

\bibitem{colombo2011modelling}
{\sc R.~M. Colombo, P.~Goatin, and M.~D. Rosini}, {\em On the modelling and
  management of traffic}, ESAIM Math. Model. Numer. Anal., 45 (2011),
  pp.~853--872.

\bibitem{santo2019capacity}
{\sc E.~Dal~Santo, C.~Donadello, S.~F. Pellegrino, and M.~D. Rosini}, {\em
  Representation of capacity drop at a road merge via point constraints in a
  first order traffic model}, ESAIM Math. Model. Numer. Anal., 53 (2019),
  pp.~1--34.

\bibitem{friedrich2020nonlocal}
{\sc J.~Friedrich, S.~G\"ottlich, and E.~Rossi}, {\em Nonlocal approaches for
  multilane traffic models}, arXiv preprint arXiv:2012.05794,  (2020).

\bibitem{friedrich2019maximum}
{\sc J.~Friedrich and O.~Kolb}, {\em Maximum principle satisfying cweno schemes
  for nonlocal conservation laws}, SIAM Journal on Scientific Computing, 41
  (2019), pp.~A973--A988.

\bibitem{friedrich2018godunov}
{\sc J.~{Friedrich}, O.~{Kolb}, and S.~{G{\"o}ttlich}}, {\em A {G}odunov type
  scheme for a class of {LWR} traffic flow models with non-local flux}, Netw.
  Heterog. Media, 13 (2018), pp.~531 -- 547.

\bibitem{garavellohanpiccoli2016book}
{\sc M.~Garavello, K.~Han, and B.~Piccoli}, {\em Models for vehicular traffic
  on networks}, vol.~9 of AIMS Series on Applied Mathematics, American
  Institute of Mathematical Sciences (AIMS), Springfield, MO, 2016.

\bibitem{GaravelloPiccoliBook}
{\sc M.~Garavello and B.~Piccoli}, {\em Traffic flow on networks}, vol.~1 of
  AIMS Series on Applied Mathematics, American Institute of Mathematical
  Sciences (AIMS), Springfield, MO, 2006.
\newblock Conservation laws models.

\bibitem{goatin2016speed}
{\sc P.~Goatin, S.~G{\"o}ttlich, and O.~Kolb}, {\em Speed limit and ramp meter
  control for traffic flow networks}, Engineering optimization, 48 (2016),
  pp.~1121--1144.

\bibitem{GoatinRossi2017}
{\sc P.~Goatin and F.~Rossi}, {\em A traffic flow model with non-smooth metric
  interaction: well-posedness and micro-macro limit}, Commun. Math. Sci., 15
  (2017), pp.~261--287.

\bibitem{GoatinScialanga}
{\sc P.~Goatin and S.~Scialanga}, {\em Well-posedness and finite volume
  approximations of the {LWR} traffic flow model with non-local velocity},
  Netw. Heterog. Media, 11 (2016), pp.~107--121.

\bibitem{greenberg2003congestion}
{\sc J.~M. Greenberg, A.~Klar, and M.~Rascle}, {\em Congestion on multilane
  highways}, SIAM J. Appl. Math., 63 (2003), pp.~818--833.

\bibitem{HautBastinChitour}
{\sc B.~Haut, G.~Bastin, and Y.~Chitour}, {\em A macroscopic traffic model for
  road networks with a representation of the capacity drop phenomenon at the
  junctions}, in Proceedings 16th IFAC World Congress, Prague, Czech Republic,
  July 2005.
\newblock Tu-M01-TP/3.

\bibitem{helbing1997modeling}
{\sc D.~Helbing and A.~Greiner}, {\em Modeling and simulation of multilane
  traffic flow}, Physical Review E, 55 (1997), p.~5498.

\bibitem{herty2003modeling}
{\sc M.~Herty and A.~Klar}, {\em Modeling, simulation, and optimization of
  traffic flow networks}, SIAM Journal on Scientific Computing, 25 (2003),
  pp.~1066--1087.

\bibitem{holden1995network}
{\sc H.~Holden and N.~H. Risebro}, {\em A mathematical model of traffic flow on
  a network of unidirectional roads}, SIAM J. Math. Anal., 26 (1995),
  pp.~999--1017.

\bibitem{HoldenRisebro2019dense}
\leavevmode\vrule height 2pt depth -1.6pt width 23pt, {\em Models for dense
  multilane vehicular traffic}, SIAM J. Math. Anal., 51 (2019), pp.~3694--3713.

\bibitem{KeimerPflug2017}
{\sc A.~Keimer and L.~Pflug}, {\em Existence, uniqueness and regularity results
  on nonlocal balance laws}, J. Differential Equations, 263 (2017),
  pp.~4023--4069.

\bibitem{keimer2019nonlocal}
\leavevmode\vrule height 2pt depth -1.6pt width 23pt, {\em Nonlocal
  conservation laws with time delay}, Nonlinear Differential Equations and
  Applications NoDEA, 26 (2019), p.~54.

\bibitem{keimer2019local}
\leavevmode\vrule height 2pt depth -1.6pt width 23pt, {\em On approximation of
  local conservation laws by nonlocal conservation laws}, J. Math. Anal. Appl.,
  475 (2019), pp.~1927--1955.

\bibitem{keimer2018nonlocal}
{\sc A.~Keimer, L.~Pflug, and M.~Spinola}, {\em Nonlocal scalar conservation
  laws on bounded domains and applications in traffic flow}, SIAM Journal on
  Mathematical Analysis, 50 (2018), pp.~6271--6306.

\bibitem{kolb2017capacity}
{\sc O.~Kolb, S.~G\"{o}ttlich, and P.~Goatin}, {\em Capacity drop and traffic
  control for a second order traffic model}, Netw. Heterog. Media, 12 (2017),
  pp.~663--681.

\bibitem{LighthillWhitham}
{\sc M.~J. Lighthill and G.~B. Whitham}, {\em On kinematic waves. {I}{I}. {A}
  theory of traffic flow on long crowded roads}, Proc. Roy. Soc. London. Ser.
  A., 229 (1955), pp.~317--345.

\bibitem{moridpour2010surveymultilane}
{\sc S.~Moridpour, M.~Sarvi, and G.~Rose}, {\em Lane changing models: a
  critical review}, Transportation Letters, 2 (2010), pp.~157--173.

\bibitem{reilly2015adjoint}
{\sc J.~Reilly, S.~Samaranayake, M.~L. Delle~Monache, W.~Krichene, P.~Goatin,
  and A.~M. Bayen}, {\em Adjoint-based optimization on a network of discretized
  scalar conservation laws with applications to coordinated ramp metering}, J.
  Optim. Theory Appl., 167 (2015), pp.~733--760.

\bibitem{Richards}
{\sc P.~I. Richards}, {\em Shock waves on the highway}, Oper. Res., 4 (1956),
  pp.~42--51.

\bibitem{RidderShen}
{\sc J.~Ridder and W.~Shen}, {\em Traveling waves for nonlocal models of
  traffic flow}, Discrete Contin. Dyn. Syst., 39 (2019), pp.~4001--4040.

\bibitem{towers2020explicit}
{\sc J.~D. Towers}, {\em An explicit finite volume algorithm for vanishing
  viscosity solutions on a network}, preprint,  (2020).

\bibitem{treiber2013traffic}
{\sc M.~Treiber and A.~Kesting}, {\em Traffic flow dynamics}, Springer,
  Heidelberg, 2013.
\newblock Data, models and simulation, Translated by Treiber and Christian
  Thiemann.

\bibitem{ZHANG2002traffic}
{\sc H.~Zhang}, {\em A non-equilibrium traffic model devoid of gas-like
  behavior}, Transportation Research Part B: Methodological, 36 (2002), pp.~275
  -- 290.

\end{thebibliography}

\end{document}